%% file: effective_putinar.tex
\documentclass[11pt, a4paper]{article}
\usepackage{amsfonts}
\usepackage[percent]{overpic}
\def\cst{\mathfrak{c}}
\input{preamble.tex}

\begin{document}
\title{On the Effective Putinar's Positivstellensatz and Moment Approximation\thanks{This work has been supported by European Union’s Horizon 2020 research and innovation programme under the Marie Skłodowska-Curie Actions, grant agreement 813211 (POEMA)}}
\author{Lorenzo Baldi\thanks{Corresponding Author: Lorenzo Baldi, Affiliation: Centre Inria
Sophia Antipolis - Méditerranée, E-mail Address: lorenzo.baldi@inria.fr}, Bernard Mourrain
  \\
{Inria at Universit\'e C\^ote d'Azur, Sophia Antipolis, France}
}
\date{\today}

\maketitle

\begin{abstract}
  We analyse the representation of positive polynomials in
  terms of Sums of Squares. We provide a quantitative version of
  Putinar's Positivstellensatz over a compact basic
  semialgebraic set $S$, with a new polynomial bound on the degree of
  the positivity certificates.  This bound involves a Łojasiewicz
  exponent associated to the description of $S$. We show that
  if the gradients of the active constraints are linearly independent on $S$ (Constraint Qualification condition), this Łojasiewicz exponent
  is equal to $1$.  We deduce the first general polynomial bound on the convergence rate of the
  optima in Lasserre's Sum-of-Squares hierarchy to the global optimum of
  a polynomial function on $S$, and the first general bound on the Hausdorff distance between the
  cone of truncated (probability) measures supported on $S$ and the
  cone of truncated pseudo-moment sequences, which are positive on the
  quadratic module of $S$.
\end{abstract}

\section{Introduction}

A fundamental question in Real Algebraic Geometry is how to describe
effectively the set of polynomials which are positive\footnote{We follow the French tradition, and call a function $f$ \emph{positive} on a domain $D$ if $f\ge 0$ on $D$ and \emph{strictly positive} on $D$ if $f>0$ on $D$.} on a given
domain.

Clearly, the set of positive polynomials on $\R^{n}$ contains the \emph{Sums of
Squares} of real polynomials (SoS).
Let $\R[\vb{X}] = \R[X_1,\dots,X_n]$ be the $\R$-algebra of
polynomials in the indeterminates $X_{1},\ldots, X_{n}$ with real
coefficients.
The convex cone of SoS polynomials
$$
\Sigma^2 \assign \Sigma^2 [\vb X] = \big\{ \, f \in \R[\vb{X}] \mid \exists r
\in \N, \ g_i \in \R[\vb{X}] \colon f = g_1^2 + \dots + g_r^2
\,\big\}
$$
is a subset of the convex cone of positive polynomials
$\pos(\R^{n})\assign\{\, p\in \R[\vb X]\mid \forall x \in \R^{n} ,
p(x)\ge 0 \,\}$.  But it is known since Hilbert \cite{hilbert_ueber_1888}, that these cones
differ: not all positive polynomials are SoS. A famous counter-example is Motzkin
polynomial \cite{motzkin_t_s_arithmetic-geometric_1967} of degree $6$ in $2$ variables, which is
positive on $\R^{2}$ but not a SoS. Such polynomials exists for
any dimension $n\ge 2$, but not in dimension $1$ since univariate
positive polynomials are SoS.

For a domain $S=\cS(\vb g) = \cS(g_1,\dots,g_r) = \{ \, x \in \R^n
\mid \ g_i(x) \ge 0 \ \textup{ for } i=1, \ldots,r \, \}$, defined by
inequalities $g_{i}\ge 0$ with $g_{i}\in \R[\vb X]$, that is, a \emph{basic closed semialgebraic set},
the set $\pos(S)$ of positive polynomials on $S$ contains
the \emph{quadratic module} generated by the tuple of polynomials $\vb g= (g_{1}, \ldots,
g_{r})$, and defined by
$$
Q = \cQ(\vb g) \assign \Sigma^2+\Sigma^2\cdot g_1+\dots+\Sigma^2 \cdot g_r
$$
and also the \emph{preordering} $O=\cO(g_1,\dots,g_r) \coloneqq \cQ (\prod_{j \in J} g_j \colon  J \subset \{ 1,\dots , r \})$.

A complete description of positive polynomials on $S$ is given by
the Krivine–Stengle Positivstellensatz:
\begin{theorem}[{\cite{krivine_anneaux_1964},\cite{stengle_nullstellensatz_1974}}] Let $\vb g$ be a tuple of polynomials and $S=S(\vb g)$. Then:
$$
\pos(S) =\{ p \in \R[\vb X] \mid \exists s\in \N, \ q_{1}, q_{2} \in
\cO(\vb g) \text{ s.t. }
q_{1} p = p^{2s} + q_{2}\}
$$
\end{theorem}
This result is an extension of Artin's theorem
\cite{artin_uber_1927}, stating that globally positive polynomials are
ratio of two SoS polynomials. But it induces a \emph{denominator} in the representation of a
positive polynomial.

Since for a general tuple $\vb g$ and $n>1$, positive polynomials on $S(\vb
g)$ do not all belong to the quadratic module $Q(\vb g)$ or even to the
preordering $\cO(\vb g)$, it is natural to ask whether the convex
cone $Q(\vb g)$ (resp. $\cO(\vb g)$) is a good inner-approximation of $\pos(S)$.
A partial answer is given by two important results due to Schmüdgen
and Putinar, and also known as denominator free
Positivstellensatz. They require the following assumption:
\begin{definition}
  \label{def::archimedean}
  Denote $\norm{\vb X}_2^2 = X_1^2 + \dots + X_n^2$. We say that a quadratic module $Q$ is \emph{Archimedean} if there exists $r \in \R$ such that $r^2-\norm{\vb X}_2^2 \in Q$.
\end{definition}
We recall these two results, which are central in the paper:
\begin{theorem}[{Schmüdgen's Positivstellensatz \cite{schmudgen_thek-moment_1991}}]
    \label{thm::schmudgen}
    Let $\cS(\vb g)$ be a compact basic semialgebraic set. Then $f>0$ on $\cS(\vb g)$ implies $f\in \cO(\vb g)$.
\end{theorem}
\begin{theorem}[{Putinar's Positivstellensatz \cite{putinar_positive_1993}}]
    \label{thm::putinar}
    Let $\cS(\vb g)$ be a basic closed semialgebraic set. If $\cQ(\vb g)$ is Archimedean, then $f>0$ on $\cS(\vb g)$ implies $f\in \cQ(\vb g)$.
\end{theorem}
As a consequence of the first result, notice that $\cS(\vb
g)$ compact implies that $\cO(\vb g)$ is Archimedean. On the other hand there are examples with $\cS(\vb
g)$ compact but $\cQ(\vb g)$ not Archimedean (see e.g.
\cite[ex.~6.3.1]{prestel_positive_2001}).

Since a positive polynomial $f\in \pos(S)$ on a compact basic
semialgebraic set $S$ can be approximated uniformly on $S$ by the
polynomial $f+\epsilon$, which is strictly positive on $S$ for
$\epsilon>0$, these results show that any $f = \lim_{\epsilon \to 0} f+\epsilon$ positive on $S$ is
the limit of polynomials in $Q(\vb g)$ (resp.
$\cO(\vb g)$).  Unfortunately, the degree of the representation of $f+\epsilon$ in $\cQ(\vb g)$ goes to infinity as $\epsilon \to 0$, see \cite{10.1006/jcom.1996.0011}.

In this paper, we provide quantitative versions of \Cref{thm::putinar}.
We give new bounds on the degree of the
representation in $Q(\vb g)$, which control the
quality of approximation of positive polynomials by elements in  $Q(\vb g)$. For this problem, also known as \emph{Effective
Putinar Positivstellensatz}, our main result is
\Cref{thm::polynomial_putinar}, which provides the first polynomial
bounds in the intrinsic parameters associated to $\vb g$ and $f$.

The proof of \Cref{thm::polynomial_putinar} is developed in \Cref{sec::ingredients} and \Cref{sec::polynomial_putinar}. In the proof and in the bound of the theorem a special role is played by the {\L}ojasiewicz exponent ${\textit{\L}}$, comparing the distance and the algebraic distance from $S$, see \Cref{def::loja}.
In \Cref{thm::L_equal_1} we prove that ${\textit{\L}} = 1$ in regular cases, i.e. when a regularity condition coming from Optimization is satisfied, see \Cref{def::CQC}. To our best knowledge this is the first analysis of the {\L}ojasiewicz exponent under regularity assumptions of any kind. Corollaries to our main results in regular cases with $\textit{\L} = 1$ are described in \Cref{cor::polynomial_putinar_regular}, \Cref{cor::polynomial_lasserre}
and \Cref{cor::moment_rate_convergence_prob}.

Polynomials whose representation in the quadratic module
$Q(\vb g)$ is of degree bounded by $2\ell$, $\ell\in \N$, are used to define a hierarchy of convex optimization problems, also known as
Lasserre's hierarchy, whose optimum value converges to the \emph{global
optimum} of a polynomial $f\in \R[\vb X]$ on $S$ under the Archimedean assumption \cite{lasserre_global_2001}.
We describe these hierarchies in \Cref{sec::lasserre_hierarchy}. In \Cref{thm::polynomial_lasserre_sos} and
\Cref{thm::polynomial_lasserre}, we deduce from \Cref{thm::polynomial_putinar} new polynomial bounds on
the convergence rate of this hierarchy to the global optimum, in terms
of the order $\ell$ of the hierarchy.

Considering the dual problem, we also analyse the quality of
approximation of measures by truncated pseudo-moment sequences used in Lasserre moment hierarchy.
In  \Cref{thm::moment_rate_convergence_prob}, we provide new bounds on
the Hausdorff distance between the cone of
 truncated probability measures (supported on $S$) and the
outer convex set of truncated positive pseudo-moment sequences of unit mass,
and on the rate of convergence when the order $\ell$ goes to
infinity. The proof of \Cref{thm::moment_rate_convergence_prob} is developed in \Cref{sec::convergence_of_moments}. The bounds involve intrinsic parameters associated to $\vb
g$ and the degree $t$ of truncation. As an intermediate step, in
\Cref{thm::moment_rate_convergence} we also bound the Hausdorff
distance between truncated positive pseudo-moment sequences and non-normalized measures.

\subsection{Truncated quadratic modules and positive polynomials}





To analyse the degree in these SoS representations, we introduce
the \emph{truncated} quadratic modules at degree (or \emph{level}) $\ell\in \N$, i.e. the polynomials in $\cQ(\vb g)$ that are generated in degree $\le \ell$:
\begin{equation}
  \label{eq::truncated_quad_module}
  \tqgen{g}{\ell} = \big\{ \, s_0 + \sum_{i=1}^r s_i g_i \mid s_j \in \Sigma^2, \ \deg s_0 \le \ell, \ \deg s_i g_i \le \ell \ \forall i =1, \dots r \,\big\} \subset \cQ(\vb g)\cap \RRg_\ell.
\end{equation}
where $\R[\vb{X}]_\ell$ is the vector space of
polynomials of degree $\le \ell$.

Effective versions of Schmüdgen and Putinar's Positivstellensatz, that give degree bounds for the representation in truncaded preorderings and quadratic modules, have been proved by Schweighofer and Nie.
\begin{theorem}[{\cite{schweighofer_complexity_2004}}]
  For all $\vb g = g_1,\dots g_r \subset \RRg = \R[X_1, \dots, X_n]$ defining $\emptyset \neq \cS(\vb g) = S \subset (-1,1)^n$ there exists $0<c \in \R$ (depending on $\vb g$ and $n$) such that, if $f \in \RRg_d$ is strictly positive on $S$ with minimum $f^* = \min_{x\in S} f(x) > 0$,
  we have $f \in \cO_\ell(\vb g)$ if \[ \ell \ge c d^2\big(1+ \big(d^2 n^d \frac{\norm{f}_{\vb X}}{f^*}\big)^c\big).\]
\end{theorem}
\begin{theorem}[{\cite{nie_complexity_2007}}]
  \label{degree_bound}
  For all $\vb g = g_1,\dots g_r \subset \RRg = \R[X_1, \dots, X_n]$ defining an Archimedean quadratic module $Q = \cQ(\vb g)$ and $\emptyset \neq \cS(\vb g) = S \subset (-1,1)^n$, there exists $0<c \in \R$ (depending on $\vb g$ and $n$) such that, if $f \in \RRg_d$ is strictly positive on $S$ with minimum $f^* = \min_{x\in S} f(x) > 0$,
  we have $f \in \tqgen{g}{\ell}$ if \[ \ell \ge c \exp(\big(d^2 n^d \frac{\norm{f}_{\vb X}}{f^*}\big)^c).\]
\end{theorem}

The norm $\norm{\cdummy}_{\vb X}$ used in \cite{schweighofer_complexity_2004} and \cite{nie_complexity_2007} is the max norm of the coefficients of the polynomial $f$ w.r.t. the weighted monomial basis $\{ \frac{\abs{\alpha}!}{\alpha_1 ! \dots \alpha_n !} \vb X ^{\alpha} \colon \abs{\alpha} \le d\}$, while the one we will use is the max norm on $\ubox$. We describe this norm and fix some notation.

\noindent{}\textbf{Notation.}
Throughout the article:
\begin{itemize}
  \item $f \in \RRg$ is a polynomial in $n$ variables of degree $d=d(f)$;
  \item $S = \cS(\vb g) = \cS(g_1, \dots, g_r)$ is the basic closed semialgebraic set defined by $\vb g = g_1,\dots,g_r$;
  \item $d(\vb g) = \max_i \deg g_i$ is the maximum degree of the inequalities defining $S$;
  \item $f^* = \min_{x \in S} f(x)$ is the minimum of $f$ on $S$, and unless otherwise stated $f^* > 0 $;
  \item $\norm{\cdummy}$ denotes the max norm of a polynomial on $\ubox$, i.e. $\norm{h} = \max_{x \in \ubox} \abs{h(x)}$;
  \item $\epsilon (f) = \frac{f^*}{\norm{f}}$ is a measure of how close is $f$ to have a zero on $S$.
\end{itemize}
For convenience we will prove our theorem in a normalized setting.

\noindent{}\textbf{Normalisation assumptions.}
In the following, we assume that
\begin{equation}
\label{assum::norm}
\begin{array}{ll}
\bullet & 1-\norm{\vb X}_2^2 \in \cQ(\vb g), \hspace{10cm}\ \\
\bullet & \norm{g_i} \le \frac{1}{2} \ \forall i \in \{1, \dots, r\}.
\end{array}
\end{equation}
We can always be in this setting by a change of variables if we start with an Archimedean quadratic module: if $r^2 - \norm{\vb X}^2 \in \cQ(\vb g)$ then $1-\norm{\vb X}^2 \in  \cQ(\vb g(r \vb X))$ (i.e. the quadratic module generated by $g_i(r X_1, \dots, r X_n)$).
By replacing $g_i$ with $\frac{g_i}{2\norm{g_i}}$, we can also assume without loss of
generality that the second point is satisfied.

The main result of the paper is the following theorem. It is an effective, general version of Putinar's Positivstellensatz with polynomial bounds for fixed $n$.
\begin{theorem}
  \label{thm::polynomial_putinar}
Assume $n \ge 2$ and let $g_{1},\ldots, g_{r}\in \R[\vb X] = \R[X_1,\dots,X_n]$ satisfying  the normalization
assumptions \eqref{assum::norm}. Let $f\in \R[\vb X]$ such that
 $f^* = \min_{x \in S} f(x) > 0$. Let $\cst$, $\textit{\L}$ be the Łojasiewicz
 coefficient and exponent given by \Cref{def::loja}. Then $f \in \cQ_\ell(\vb g)$ if
  \begin{align*}
  \ell & \ge  O(
  n^{3}
  2^{5n{\textit{\L}}}
  r^{n}
  \cst^{2n}
  d(\vb g)^{n}
  d(f)^{3.5n{\textit{\L}}}
  \epsilon(f)^{-{2.5 n {\textit{\L}}}})\\
  & = \gamma(n, \vb g)\, d(f)^{3.5n{\textit{\L}}}
  \epsilon(f)^{-{2.5 n {\textit{\L}}}},
  \end{align*}
  where $\gamma(n,\vb g) \ge 1$ depends only on $n$ and $\vb g$.
\end{theorem}
Notice that the only parameters in the bound that depend on $f$ are $d(f)$ and $\epsilon(f)$. We also remark that exponents in \Cref{thm::polynomial_putinar} have been simplified for the
sake of readability and are not optimal: see \Cref{eq::sharp exp} for sharper bounds, especially for the case $n \gg 0$. Moreover we remark that the assumption $n \ge 2$, only used to do this simplification, is not a serious limitation since the univariate case is already well studied, see for instance \cite{powers_polynomials_2000}.

In the definition of $\epsilon(f)$ we use the max norm $\norm{\cdummy}$ on $\ubox$ instead of $\norm{\cdummy}_{\vb X}$ used in  \cite{nie_complexity_2007}, because it does not depend on the choice of a basis and on the representation of the polynomials. 
However, for polynomials of bounded degree, the two norms  are equivalent. Using \cite[lem.~7]{nie_complexity_2007} to express our bound with $\norm{\cdummy}_{\vb X}$ would result in an extra factor $2^{2.5 n {\textit{\L}}} n^{2.5 d(f) n {\textit{\L}}} d(f)^{2.5 n {\textit{\L}}}$, while keeping the exponent of $\epsilon(f)$.

In \Cref{sec::polynomial_putinar} we develop in detail the proof of
\Cref{thm::polynomial_putinar}. The ingredients for the proof are
introduced in \Cref{sec::ingredients}. The main differences with the
one of \cite{nie_complexity_2007} is the use of an effective
Schmüdgen's Positivstellensatz on the unit box \cite{laurent_effective_2021},
and an effective approximation of regular functions on the unit interval, see \Cref{thm:chebyshev}.
Moreover in \Cref{subsec::delta} we prove that the main exponent of the bound, i.e. the Łojasiewicz exponent ${\textit{\L}}$, is equal to $1$ for regular polynomial optimization problems, see \Cref{def::CQC} and \Cref{thm::L_equal_1}. The corollary of \Cref{thm::polynomial_putinar} in these regular cases is \Cref{cor::polynomial_putinar_regular}.

As a corollary of \Cref{thm::polynomial_putinar} we analyse the
convergence of
Lasserre hierarchies used in polynomial optimization. In \Cref{sec::lasserre_hierarchy} we focus on the optimum of the hierarchy, proving in \Cref{thm::polynomial_lasserre_sos} and \Cref{thm::polynomial_lasserre} a new, general polynomial convergence as corollary of our main result (see \Cref{cor::polynomial_lasserre} for regular Polynomial Optimization Problems).
On the other hand in \Cref{sec::convergence_of_moments} we focus on the convergence of the feasible truncated pseudo-moment sequences of the moment hierarchy to truncated moment sequences of measures supported on $S$: we prove in \Cref{thm::moment_rate_convergence_prob} that we can bound their Hausdorff distance using \Cref{thm::polynomial_putinar}.

\subsection{Truncated pseudo-moment sequences and measures}
\label{subsec::truncated_moments}
Dualizing our point of view, we consider
the convex cone $\cM(S)$ of Borel measures supported on $S$, which is dual to $\pos(S)$. We denote by $\cMone (S)$ the set of Borel
\emph{probability} measures supported on $S$.

The dual of polynomials is described as follows (see \cite{mourrain_polynomialexponential_2018} for more details). For $L \in (\R[\vb{X}])^* =\hom_{\RR}(\RRg, \R)$, we denote $\braket{L}{f} = L (f)$ the application of $L$ to $f \in \R[\vb X]$, to emphasize the dulity pairing between $\RRg$ and $(\R[\vb{X}])^*$. Recall that $(\R[\vb{X}])^*\cong \R[[\vb{Y}]] \coloneqq \R[[Y_1,\dots,Y_n]]$, with the isomorphism given by:
\[
    (\R[\vb{X}])^* \ni L \mapsto \sum_{\alpha \in \N^n} \braket{L}{\vb X^{\alpha}} \frac{\vb Y^{\alpha}}{\alpha !} \in \R[[\vb{Y}]],
\]
where $\{\frac{\vb Y^{\alpha}}{\alpha !}\}$ is the dual basis\footnote{To be more precise, \emph{basis} means here a Schauder basis of $\R[[\vb{Y}]]$ equipped with the $ (Y_1, \dots Y_n)$-adic topology.} to $\{\vb X^{\alpha} \}$, i.e. $\braket{\vb Y^{\alpha}}{\vb X^{\beta}}=\alpha !\, \delta_{\alpha,\beta}$. With this basis we can also identify $L \in (\R[\vb{X}])^*$
 with its sequence of coefficients (\emph{pseudo-moments} of $L$) $(L_{\alpha})_{\alpha}$, where $L_{\alpha} = \braket{L}{\vb X^{\alpha}}$.

 Among all the linear functionals of special importance are the ones
 \emph{coming from a measure}, i.e. $L \in \RRg^*$ such that
 there exists a Borel measure $\mu\in \cM(S)$ with $\braket{L}{f} = \int f \dd{\mu}$ for all $f \in \RRg$. In this case the sequence $(\mu_{\alpha})_{\alpha}$ associated with $\mu$ is the sequence of \emph{moments}: $\mu_{\alpha}=\int X^{\alpha} \dd{\mu}$. We are interested in the case when $S$ is compact. In such a case the moment problem is \emph{determinate}, i.e. the sequence of moments $(\mu_{\alpha})_{\alpha}$ determines uniquely $\mu \in \cM(S)$, see for instance \cite[ch.~14]{schmudgen_moment_2017}. Therefore we will identify $\mu$ with its associated linear functional
acting on polynomials (or equivalently with its sequence of moments),
so that $\cM(S) \subset (\R[\vb{X}])^*$.

We will work in the \emph{truncated} setting, i.e. when we restict our
linear functionals to a fixed, finite dimensional subspace of
$\RRg$. In particular we denote $(\cdot)^{[t]}$ the restriction of a
linear functional (or of a family of linear functionals) to $\RRg_t$,
i.e. to polynomials of degree at most $t$. In coordinates, if $L
= (L_{\alpha})_{\abs{\alpha} \le d} \in \RRg_d^*$ then
$L^{[t]} = (L_{\alpha})_{\abs{\alpha} \le t} \in \RRg_t^*$, i.e. $L^{[t]}$ is the truncation of the pseudo-moment sequence to degree $t$.

We are interested in the dual algebraic objects to truncated quadratic modules: the \emph{truncated positive linear functionals}
\[
    \cL_\ell(\vb g) = \{ \, L \in (\R[\vb X]_\ell)^* \mid \forall q \in \cQ_\ell(g) \  \braket{L}{q}\ge 0 \, \} = \tqgen{g}{\ell}^{\vee},
\]
i.e. $\cL_\ell(\vb g)$ is the dual convex cone to $\cQ_\ell(\vb g)$. See \cite{rockafellar_convex_1997} for more about convex cones and convex duality, and \cite{grigoriy_semidefinite_2012} for their use in Optimization and Convex Algebraic Geometry.
We define the affine section
\[
    \cLone_\ell(\vb g) = \big\{ \, L \in \cL_\ell(\vb g) \mid \braket{L}{1}=1 \, \big\}.
\]
Let $t = \floor{\frac{\ell}{2}}$. We verify that for $L \in \cL_\ell(\vb g)$, $\braket{L}{1} = 0$ implies $L^{[t]} = 0$, in order to prove that $\cLone_\ell(\vb g)^{[t]}$ is a generating section of $\cL_\ell(\vb g)^{[t]}$.
Assume that $\braket{L}{1}=0$.
For all $h \in \RRg_t$ and $x \in \R$, we have $0 \le \braket{L}{(1+x h)^2} = \braket{L}{1} + 2x\braket{L}{h} + x^2 \braket{L}{h^2}$. If $\braket{L}{1} = 0$, then the polynomial $x \mapsto 2x\braket{L}{h} + x^2 \braket{L}{h^2}$ is positive on $\R$ and has a zero at $x =0$. Thus $x=0$ is a double root and $\braket{L}{h} = 0$. This implies that $L$ restricted to polynomials of degree $\le t$ is zero, i.e. $L^{[t]} = 0$. Therefore if $L^{[t]} \neq 0$ then $\braket{L}{1}>0$ and $L^{[t]} = \braket{L}{1} \frac{L^{[t]}}{\braket{L}{1}}$, with $\frac{L^{[t]}}{\braket{L}{1}}\in \cLone_{\ell}(\vb g)^{[t]}$. This shows that $\cLone_\ell(\vb g)^{[t]}$ is a generating section of $\cL_\ell(\vb g)^{[t]}$.

Truncated positive linear functionals are an outer approximation of measures supported on $S$. They are used in Polynomial Optimization Problems (POP) to compute lower approximations of the minimum of a polynomial function $f$ on $S$, see \Cref{sec::lasserre_hierarchy}. Under the Archimedean assumption, convergence to measures of the linear functionals realizing the lower approximations have been proved in \cite[th.~3.4]{schweighofer_optimization_2005} for POP.

However nothing is said about the rate of convergence. We use \Cref{thm::polynomial_putinar}, quantifying how good is the inner approximation of positive polynomials by truncated quadratic modules, to answer the question we are interested in: how good is the outer approximation of (probability) measures by truncated positive linear functionals (of total mass one)? \Cref{thm::moment_rate_convergence_prob} gives the answer. In the theorem we bound the Hausdorff distance $\hdist{\cdummy}{\cdummy}$
between the outer approximation and the measures supported on $S$, where $\displaystyle \hdist{A}{B} = \max \big\{\sup_{a\in A} \dist{a}{B}, \ \sup_{b\in B} \dist{A}{b} \big\}$.

\begin{theorem}\label{thm::moment_rate_convergence_prob}
  Assume $n \ge 2$ and that the normalisation assumptions \eqref{assum::norm} are satisfied, and in particular that $1 - \norm{\vb X}_2^2 = q \in \cQ_{\ell_0}(\vb g)$. Let $0 <
  \epsilon \le {1\over 2}$, $t\in \N_{+}$ and $\ell\in \N$ such that $\ell \ge \gamma(n,\vb g)\, 6^{2.5n\textit{\L}}\, t^{6n\textit{\L}}\, \binom{n+t}{t}^{2.5n\textit{\L}}
  \epsilon^{-{2.5n\textit{\L}}}$ and $\ell\ge 2t+\ell_{0}$, with $\gamma(n, \vb g)$ given by \Cref{thm::polynomial_putinar}. Then
   \[
     \hdist{\cMone(S)^{[2t]}}{\cLone_\ell(\vb g)^{[2t]}}\le \epsilon.
   \]
\end{theorem}

The proof of \Cref{thm::moment_rate_convergence_prob} is developed in \Cref{sec::convergence_of_moments}.
The corollary of \Cref{thm::L_equal_1} in regular cases with $\textit{\L}=1$ is \Cref{cor::moment_rate_convergence_prob}.

\subsection{Related works}
Complexity analysis in Real Algebraic Geometry is an active area of
research, where obtaining good upper bounds is challenging. See for
instance \cite{lombardi_elementary_2020} for elementary recursive
degree bounds in Kivrine-Stengle Positivstellensatz, and
\cite{safey_el_din_complexity_2018} for computation complexity of real
radicals.
Among all the Stellens\"atzen, we consider Putinar's
Positivstellensatz, which allows a denominator free representation of strictly positive polynomials and
has well-know applications in Polynomial Optimization. The representation of strictly positive polynomials has a long history. For instance P\"olya \cite{polya_uber_1928} gave a representation of homogeneous polynomials $f$ strictly positive on the simplex $\Delta$ as ratio of a polynomial with positive coefficients and $\norm{\vb X}_2^{k}$, for some $k$. It is interesting to notice that, although no explicit degree bounds were presented, the degree of the representation depends on the sup norm of $f$ on $\Delta$ and on its minimum $f^* > 0$, i.e. on $\epsilon(f)$, in analogy with \Cref{thm::polynomial_putinar}. Another representation for homogeneous polynomials has been proved by Reznick \cite{Reznick1995UniformDI}, where it is shown that an homogeneous polynomial $f$ strictly positive on $\R^n \setminus \{ 0 \}$ (a \emph{positive definite form}) can be written as ratio of even powers of linear forms and $\norm{\vb X}_2^{k}$, for some $k$. Degree bounds for the representation are provided, and again we find a dependence on $\epsilon(f)$ (defined restricting $f$ to the $n-1$ hypersphere) with exponent equal to $-1$.

A general, effective version of Putinar's Positivstellensatz have been
proved in \cite{nie_complexity_2007} (see also
\cite{schweighofer_complexity_2004} for a general effective
Schmüdgen's Positivstellensatz).
This result is used in \cite{magron_exact_2021} to give bounds on
the degree of rational SoS positivity certificates, which are
exponential in the bit-size of the input polynomials $f,\vb g$.
Compared to \cite{nie_complexity_2007}, \Cref{thm::polynomial_putinar} gives
degree bounds, which are polynomial and
not exponential in $\epsilon(f)$. This implies a polynomial rate
convergence of Lasserre hierarchies, see
\Cref{thm::polynomial_lasserre_sos}, and not logarithmic as in
\cite{nie_complexity_2007}. For special semialgebraic sets, the bounds
on the convergence rate can be improved: see for instance
\cite{laurent_effective_2021} for convergence on the unit box and
\cite{fang_sum--squares_2020} for the unit sphere. The convergence
rate of the upper bounds of Lasserre SoS density hierarchy over the sphere
has been studied in \cite{de_klerk_convergence_2019}.

The proof of \Cref{thm::polynomial_putinar} is based on the
construction of a perturbation polynomial $q \in \cQ(\vb g)$ and the reduction to a simpler
semialgebraic set. This construction of the perturbation polynomial
$q$ using univariate SoS, has already been used in
\cite{schweighofer_complexity_2004},
\cite{schweighofer_optimization_2005}, \cite{nie_complexity_2007},
\cite{averkov_constructive_2013},
\cite{kurdyka_convexifying_2015}. In \cite{MAI2022101663}  Mai and Magron investigate with a similar technique the representation of strictly positive polynomials on arbitrary semialgebraic sets as ratio of polynomials in the quadratic module and $(1+\norm{\vb X}^2_2)^k$ for some $k$, giving degree bounds for the representation. These bounds are polynomial on $f^*$ (and thus on $\epsilon(f)$), but the exponent and the constant are not explicit in the general case. Moreover they remark that they were not able to derive a polynomial Effective Putinar's Positvstellensatz using their perturbation polynomials, defined recursively.

Our main improvements in the proof
are the generalisation from univariate SoS or recursively defined perturbation polynomials to a positive polynomial
echelon function, see \Cref{subsec::echelon}, and the use of an
Effective Schmüdgen's Positivstellensatz on the unit box from
\cite{laurent_effective_2021}.  Moreover in \Cref{subsec::delta} we
analyse regular cases that result in very simple exponents, see
\Cref{cor::polynomial_putinar_regular}. \Cref{cor::polynomial_putinar_regular} can be applied in particular in the case of a single ball constraint, that was analysed in \cite{MAI2022101663} for the Putinar-Vasilescu's Positivstellensatz, that introduces a denominator: the exponent in this case is equal to $-65$, while \Cref{cor::polynomial_putinar_regular} gives $-2.5n$. 
We conjecture that it is possible to remove the dependence on $n$ in the exponent of the Effective Putinar's Positivstellensatz.

This approach with a perturbation polynomial $q$ has also been used in
\cite{kurdyka_convexifying_2015} to prove a Weierstrass Approximation
theorem on compact sets for positive polynomials, where the
approximation is done with polynomials in the quadratic module
$\cQ(\vb g)$. We obtain an equivalent result with our polynomial
echelon functions in \Cref{thm::quad_module_weirstrass} with bounds on
the degree of $q$.

Convergence of pseudo-moments sequences to measures in Lasserre's hierarchies
has been studied in \cite{schweighofer_optimization_2005} for
Polynomial Optimization Problems and more generally in
\cite{tacchi_convergence_2021} for Generalized Moment Problems (GMP).
The convergence rates of moment hierarchies in GMP over the simplex and
the sphere have been studied in \cite{kirschner_convergence_2021}.
To our best knowledge there is no analysis of the convergence rate
for general compact basic semialgebraic sets in
the literature. In \Cref{thm::moment_rate_convergence_prob} we prove
such a rate of convergence for the pseudo-moment sequences used in Polynomial
Optimization, deducing this speed from \Cref{thm::polynomial_putinar}.

\section{The ingredients of the proof}
\label{sec::ingredients}
To prove the polynomial bound
for the Effective Putinar's
Positivstellensatz
(\Cref{thm::polynomial_putinar}), we proceeds as follows, refining the approach in
\cite{schweighofer_optimization_2005}, \cite{nie_complexity_2007},
\cite{averkov_constructive_2013}:
\begin{itemize}
 \item We perturb $f$ into a polynomial $p$ such that $p$ is strictly positive
   on the box $\ubox$ and $f-p$ is in the quadratic module $\cQ(\vb g)$;
 \item We compute an SOS representation of $p$ in $\cQ(1 -\norm{\vb
     X}_2^{2})$ to deduce the representation of $f \in \cQ(\vb g)$.
\end{itemize}
Notice that if $f>0$ on $\ubox$ then we can directly apply \Cref{ass::putinar_on_B} and \Cref{lem::box_to_ball} to conclude the proof. Therefore in the following we always assume that there exists $x\in \ubox \setminus S$ such that $f(x) \le 0$.

To compute the perturbed polynomial $p$, we use a univariate
echelon-like polynomial, which shape and degree depends on
the distance between a level-set of $f$ and $S$ and on a lower bound of
the algebraic distance to $S$.
We detail these ingredients hereafter.

\subsection{Distance between level sets of $f$ and $S$}
\label{subsec::dist_AS}

We define the complementary in $\ubox$ of a neighbourhood of $S$,
where $f$ is strictly smaller than $f^*$:\[A = \{ \, x \in \ubox \mid f(x)
  \le \frac{3f^*}{4} \,\}\]
that is a sublevel set of the function $f$.

We are going to bound the distance from $A$ to $S$ in terms of $\epsilon(f)$.
We recall first a Markov inequality theorem, bounding the norm of the
gradient of a polynomial function on a convex body, in the special case of the box $\ubox$.

\begin{theorem}[{\cite[th.~3]{kroo_bernstein_1999}}]
\label{thm::markov_inequality}
  Let $p \in \RRg_d$ be a polynomial of degree $\le d$. Then:
  \[
    \norm{\norm{\grad p}_2} \le (2d^2-d) \norm{p}.
  \]
\end{theorem}
Recall that the Lipschitz constant $L_f$ of $f$ is the smallest real number such that $\abs{f(x) - f(y)} \le L_f \norm{x-y}_{2}$ for all $x,y$ in the domain of $f$.
Using \Cref{thm::markov_inequality} to bound the Lipschitz constant of
$f$ on $\ubox$, we can lower bound the distance between $A$ and $S$.
\begin{proposition}\label{prop::distance_bound}
  Let $A$ and $S$ be as above. Then $\hdist{A}{S} \ge \frac{\epsilon(f)}{8d^2}$.
\end{proposition}
\begin{proof}
  We first relate the Lipschitz constant $L_f$ of $f$ on $\ubox$ with $\norm{f}$.

  From the mean value theorem we deduce that for all $x, y \in \ubox$ we have $\abs{f(x) - f(y)} \le \norm{\norm{\grad f}_2} \norm{x-y}_{2}$. Then from the definition of Lipschitz constant and \Cref{thm::markov_inequality}:
  \begin{equation}
    \label{eq::lips}
    L_f \le \norm{\norm{\grad f}_2} \le (2d^2-d)\norm{f} \Rightarrow \frac{1}{L_f} \ge \frac{1}{(2d^2-d)\norm{f}}.
  \end{equation}

  Now let $x \in A$ and $y \in S$. By definition of $L_f$ we have $\abs{f(x)-f(y)} \le L_f\norm{x-y}_{2}$. Since $x \in A$ we have $f(x) \le \frac{3f^*}{4}$; since $y \in S$ we have $f(y) \ge f^*$: thus $\abs{f(x)-f(y)} \ge \frac{f^*}{4}$ and  $\norm{x-y}_{2} \ge \frac{f^*}{4L_f}$.
  As the inequality hold for all $x \in A$ and $y \in S$ we can use \Cref{eq::lips} to conclude:
  \[
    \hdist{A}{S} \ge \frac{f^*}{4L_f} \ge \frac{f^*}{4
      (2d^2-d)\norm{f}} = \frac{\epsilon(f)}{4(2d^2-d)}\ge \frac{\epsilon(f)}{8d^2} .
  \]
 \end{proof}
\subsection{Bounds on the algebraic distance to $S$}
\label{subsec::delta}
The algebraic distance to the set $S$ is the continuous semialgebraic
function defined by
\[
   G(x) =  \abs{\min\{g_{1}(x), \ldots, g_r(x), 0\}}.
\]
Clearly, $G(x)=0$ if and only if $x\in S$, and $G(x)>0$ if $x\not\in
S$.
We are going to bound from below the function $G$ on $A$, that is
find $\delta \in \RR_{>0}$ such that
\begin{equation}
  \label{eq::delta}
 \forall x \in A, \ G(x) \ge \delta
\end{equation}
(such a $\delta$ exists since $A$ is compact and $G(x)>0$ on $A$).

To express such a $\delta$ in terms of $\epsilon(f)$, we use
Łojasiewicz inequalities, introduced by Łojasiewicz in \cite{lojasiewicz_1959}, following
and expanding the approach in \cite[lem.~13]{nie_complexity_2007}.
\begin{theorem}[{\cite{lojasiewicz_1959}, \cite[cor.~2.6.7]{bochnak_real_1998}}]
\label{thm::lojasiewicz}
  Let $B$ be a closed and bounded semialgebraic set and let $f, g$ be two continuous semialgebraic functions from $B$ to $R$ such that $f^{-1}(0) \subset g^{-1}(0)$. Then there exists $c, \textit{\L} \in \RR_{>0}$ such that $\forall x \in B$:
  \[
    \abs{g(x)}^{\textit{\L}} \le c \abs{f(x)}.
  \]
\end{theorem}
We use now \Cref{thm::lojasiewicz} and \Cref{prop::distance_bound} to bound $\delta$ in terms of $\epsilon(f)$.

\begin{definition}
\label{def::loja}
Let $\cst, \textit{\L}$ be the constant and exponent of
  Łojasiewicz inequalities (\Cref{thm::lojasiewicz}) for the functions
$G: x\in  \ubox \mapsto G(x) = \abs{\min\{g_{1}, \dots, g_r(x), 0\} }$
and  $D: x\in \ubox \mapsto D(x) = \dist{x}{S}$, that is, for $x\in \ubox$
\begin{equation}\label{eq::loja}
D(x)^{\textit{\L}} \le \cst\, G(x).
\end{equation}
\end{definition}
These constant and exponent are well-defined by
\Cref{thm::lojasiewicz}, since the functions $D$ and $G$ are continuous semialgebraic
and $S = D^{-1}(0) = G^{-1}(0)$.
\begin{lemma}
\label{lem::delta}
  We can choose $\delta =
  \frac{1}{\cst} \left( \frac{\epsilon(f)}{8d^2}\right)^{\textit{\L}}$ in \Cref{eq::delta}, where $\cst, \textit{\L}$
  are defined in \Cref{def::loja}.
\end{lemma}
\begin{proof}
By \Cref{prop::distance_bound}, we have $D(x)=\dist{x}{S} \ge
\frac{\epsilon(f)}{8d^2}$ for $x\in A$. Then from \Cref{eq::loja}, we deduce
that for $x\in A$,
\[
  \left( \frac{\epsilon(f)}{8d^2}\right)^{\textit{\L}} \le  \cst G(x) \Rightarrow G(x) \ge \frac{1}{\cst} \left( \frac{\epsilon(f)}{8d^2}\right)^{\textit{\L}}.
\]
Therefore, we can choose $\delta = \frac{1}{\cst} \left( \frac{\epsilon(f)}{8d^2}\right)^{\textit{\L}}$.
\end{proof}

The exponent $\textit{\L}$ in \Cref{def::loja} will play an important role in
the bounds of the Effective Putinar's Positivstellensatz. We show now that, under generic assumptions, we can choose $\textit{\L}{} = 1$, as suggested in the following example.
\begin{example}
\label{ex::L_equal_1}
  Consider the univariate polynomial $g(X) = {\epsilon^2 - X^2}$ and let $S = S(g) = [-\epsilon, \epsilon] \subset [-1,1]$. Now let $x \in [-1,1]$ and $D, G$ be as in \Cref{def::loja}. It is easy to show that:
  \[
    D(x) \le \frac{1}{2\epsilon} G(x).
  \]
  Indeed, if for example $\epsilon \le x \le 1$, we have $D(x) = x -\epsilon$ and $G(x) = x^2 - \epsilon^2 = (x+\epsilon)(x-\epsilon)$ and $D(x) = \frac{1}{x+\epsilon} G(x) \le \frac{1}{2\epsilon}G(x)$. This shows that we can choose $\textit{\L} = 1$ for all $\epsilon > 0$.

  On the other hand if $\epsilon = 0$, i.e. $g(X) = -X^2$ and $S = \{
  0\}$, we have a singular equation. A simple computation shows that
  it is not possible to choose $\textit{\L} = 1$ in this case. The minimum $\textit{\L}$ satisfying the inequality is $\textit{\L}=2$.
\end{example}

We introduce a regularity condition needed to prove $\textit{\L} = 1$, generalizing \Cref{ex::L_equal_1}. This is a standard condition in optimization (see \cite[sec.~3.3.1]{bertsekas_nonlinear_1999}), which implies the so-called Karush–Kuhn–Tucker (KKT) conditions \cite[prop.~3.3.1]{bertsekas_nonlinear_1999}.

\begin{definition}
\label{def::CQC}
Let $x \in \cS(\vb g)$. The active constraints at $x$ are the constraints $g_{i_{1}}, \ldots, g_{i_{m}}$ such that $g_{i_{j}}(x)=0$.
We say that the \emph{Constraint Qualification condition} (CQC) holds at $x$ if for the active constraints $g_{i_{1}}, \ldots, g_{i_{l}}$ at $x$, the gradients $\nabla g_{i_{1}}(x), \ldots,  \nabla g_{i_{m}}(x)$ are linearly independent.
\end{definition}

\begin{lemma}
\label{lem::difference_as_gradient}
  Let $y \in \RR^n \setminus \cS(\vb g)$, and let $z$ be a point in $S = \cS(\vb g)$ minimizing the distance of $y$ to $S$, that is $\dist{y}{S} = \norm{y-z}_2$. If $\{\,g_i \colon i \in I \,\}$ are the active constraints at $z$ and the CQC holds, then there exist $\lambda_i \in \RR_{\ge 0}$ such that:
  \[
    y - z = \sum_{i\in I} \lambda_i \grad (-g_i)(z).
  \]
\end{lemma}
\begin{proof}
  Fix $y \in \RR^n$. Notice that $y - x = - \frac{\grad \norm{y-x}_2^2}{2}$, where the gradient is take w.r.t. $x$. Moreover $z \in S$ such that $\dist{y}{S} = \norm{y-z}_2$ is a minimizer of the following Polynomial Optimization Problem:
  \[
    \min_x \frac{\norm{y-x}_2^2}{2} \colon g_i(x) \ge 0 \ \forall i\in \{1,\dots,r \}.
  \]
  Since the CQC holds at $z$, we deduce from \cite[prop.~3.3.1]{bertsekas_nonlinear_1999} that the KKT conditions hold. In particular:
  \[
    \frac{\grad \norm{y-z}_2^2}{2} = \sum_{i\in I} \lambda_i \grad g_i(z)
  \]
  For some $\lambda_i \in \RR_{\ge 0}$. Therefore $y-z = - \frac{\grad \dist{y}{z}^2}{2} = \sum_{i\in I} \lambda_i \grad (-g_i)(z)$.
\end{proof}
We first fix a point $z \in S$ and consider the points $y$ such that the closest point to $y$ on $S$ is $z$. We prove that $\textit{\L} = 1$ in the sector of these $y$ where all the active constraints at $z$ are strictly negative at $y$.
\begin{lemma}
\label{lem::L_equal_1_local_negative}
  Let $D, G$ as in \Cref{def::loja} and let $z \in \cS(\vb g)$ with active constraints $g_i \colon i \in I$. Then there exists $\epsilon' = \epsilon'(z)>0$ and constant $\cst' = \cst'(z)>0$ such that for all $y$ with:
  \begin{itemize}
    \item $D(y) = \dist{y}{S} = \norm{y-z}_2$;
    \item $D(y) \le \epsilon'$;
    \item $g_i(y) < 0 $ for all $i\in I$,
  \end{itemize}
  we have $D(y) \le \cst' G(y)$.
\end{lemma}
\begin{proof}
Let $z \in S$ and $y \in \R^n$ be such that $D(y) = \norm{y-z}_2$. Consider the Taylor expansion of $g_i$ at $z$ for the active constraints $\{\,g_i \colon i \in I \,\}$: there exists $\vb h_i (y) = (h_{i,1}(y),\dots,h_{i,n}(y))$, where $\lim_{y \to z} h_{i,j}(y) = 0$, such that:
\begin{equation}
\label{eq::taylor_exp}
  g_i (y) = \grad g_i(z) \cdot (y-z) + \sum_{j=1}^n h_{i,j}(y)(y_i-z_i) = \grad g_i(z) \cdot (y-z) + \vb h_i(y) \cdot (y-z).
\end{equation}
In other words, the $\vb h_i(y) \cdot (y-z)$ is the remainder of the first order Taylor approximation.
Since the CQC is satisfied at $z$ we can apply \Cref{lem::difference_as_gradient}: there exists $\lambda = (\lambda_i \colon i \in I)$ with $y - z = \sum_{i\in I} \lambda_i \grad (-g_i)(z)$. Substituting we obtain $\forall i \in I$:
\[
  g_i (y) = - \sum_{j\in I} \lambda_j (\grad g_i(z) \cdot \grad g_j(z)) + \vb h_i(y) \cdot (y-z).
\]
We denote:
\begin{itemize}
  \item $\vb g_I (y) = (g_i(y) \colon i \in I)$;
  \item $\vb J_I(z) = \jac(\vb g_I)(z) = (\grad g_i(z))_{i\in I}$ the Jacobian matrix;
  \item $\vb N_I(z) = (\grad g_i(z) \cdot \grad g_j(z))_{i,j} = \vb J_I(z)^t \vb J_I(z)$ the Gram matrix of $\grad g_i(z)$; and 
  \item $\vb h(y) = (h_{i,j}(y))_{i,j}$.
\end{itemize}
With this notation we get:
\[
  \vb g_I(y) = - \vb N_I(z) \lambda + \vb h(y) (y-z).
\]
Since CQC hold at $z$, the $\grad g_i(z)$ are linearly independent and thus $\vb N_I(z)$ is invertible. Indeed, if $\vb N_I(z)$ is not invertible then there exists $0 \neq v \in \R^{\abs{I}}$ such that $\vb N_I(z) v = 0$. Therefore $0 = v^t \vb N_I(z) v = (\vb J_I(z) v)^t \vb J_I(z)v = \norm{\vb J_I(z)v}_2^2$. Hence $\vb J_I(z)v =0$, contradicting the linear independence of $\grad g_i(z) \colon i \in I$. Thus we can solve for $\lambda$:
\begin{equation}
  \label{eq::lambda}
  \lambda = - \vb N_I(z)^{-1} \vb g_I(y) + \vb N_I(z)^{-1} \vb h (y) (y-z).
\end{equation}
Recall from \Cref{lem::difference_as_gradient} that we have:
\[
  y - z = \sum_{i\in I} \lambda_i \grad (-g_i)(z) = -\vb J_I (z) \lambda.
\]
 Substituting $\lambda$ from \Cref{eq::lambda} we obtain:
\[
  y - z = \vb J_I (z) \vb N_I(z)^{-1} \vb g_I(y) - \vb J_I (z) \vb N_I(z)^{-1} \vb h (y) (y-z).
\]
Taking the norm we deduce that:
\begin{align*}
    \norm{y - z}_2 &= \norm{\vb J_I (z) \vb N_I(z)^{-1} \vb g_I(y) - \vb J_I (z) \vb N_I(z)^{-1} \vb h (y) (y-z)}_2 \\ &\le \norm{\vb J_I (z)}_2 \norm{\vb N_I(z)}_2^{-1} \norm{\vb g_I(y)}_2 + \norm{\vb J_I (z)}_2 \norm{\vb N_I(z)}_2^{-1} \norm{\vb h (y)}_2 \norm{y-z}_2
\end{align*}
where $\norm{\cdummy}_2$ denotes the 2-norm (resp. operator norm) of vectors (resp. matrices). Therefore:
\begin{equation}
  \label{eq::inequality}
  \big(1-\norm{\vb J_I (z)}_2 \norm{\vb N_I(z)}_2^{-1} \norm{\vb h (y)}_2\big)\norm{y-z}_2 \le \norm{\vb J_I (z)}_2 \norm{\vb N_I(z)}_2^{-1} \norm{\vb g_I(y)}_2
\end{equation}
Notice that:
\begin{itemize}
  \item $\norm{\vb g_I (y)}_2 = \sqrt{\sum_{i\in I} g_i(y)^2} \le \sqrt{\abs{I}}\, \max_{i\in I} \abs{g_i(y)} \le \sqrt{r}\, G(y)$, since \[G(y) = \abs{\min \{0,g_i(y) \colon i \in \{ 1,\dots , r\}\}} = \max \{ 0 , \abs{g_i(y)} \colon g_i(y) < 0 \}\]
  and $g_i(y) < 0$ for all $i \in I$ by hypothesis;
  \item $1-\norm{\vb J_I (z)}_2 \norm{\vb N_I(z)}_2^{-1} \norm{\vb h (y)}_2 \ge \frac{1}{2}$ if $y$ is close enough to $S$. Indeed $h_{i,j} \to 0$ when $y \to z$, i.e. when $\norm{y-z}_2 = \dist{y}{S} = D(y)$
  is going to zero. Thus we can choose $\epsilon'$ such that $D(y) \le
  \epsilon'$ implies $1-\norm{\vb J_I (z)}_2 \norm{\vb N_I(z)}_2^{-1} \norm{\vb h (y)}_2 \ge \frac{1}{2}$.
\end{itemize}
Then we deduce from \Cref{eq::inequality}:
\[
  D(y) = \norm{y-z}_2 \le 2 \sqrt{r} \norm{\vb J_I (z)}_2 \norm{\vb N_I(z)}_2^{-1} G(y)
\]
when $D(y) \le \epsilon'$, that proves the lemma with $\cst' = 2 \sqrt{r} \norm{\vb J_I (z)}_2 \norm{\vb N_I(z)}_2^{-1}$.
\end{proof}
We generalize the previous lemma, removing the assumption that all the active constraints are negative at $y$.
\begin{lemma}
\label{lem::L_equal_1_local}
  Let $D, G$ as in \Cref{def::loja} and assume that the CQC hold at $z \in S = \cS(\vb g)$. Then there exists $\epsilon'' = \epsilon''(z)>0$ and constant $\cst'' = \cst''(z)>0$ such that for all $y$ with:
  \begin{itemize}
    \item $D(y) = \dist{y}{S} = \norm{y-z}_2$;
    \item $D(y) \le \epsilon''$;
  \end{itemize}
  we have $D(y) \le \cst'' G(y)$.
\end{lemma}
\begin{figure}
\begin{center}
\begin{overpic}[width=7cm,grid=false,tics=10]{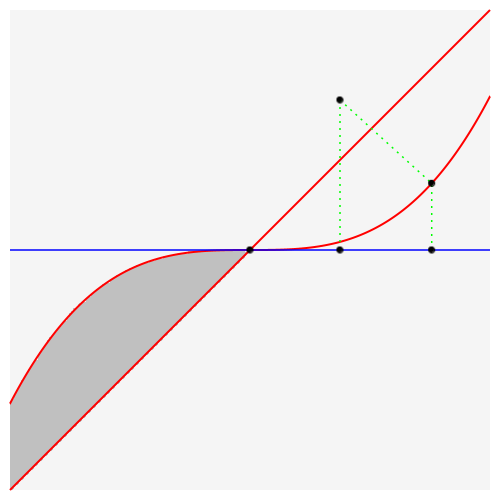}
\put (50,45) {$z$}
\put (67,84) {$y$}
\put (67,45) {$z'$}
\put (85,67) {$z''$}
\put (85,45) {$z'''$}
\put (20,35) {$S(\vb g)$}
\put (4.5,30) {$g_1 = 0$}
\put (15,19) {$g_2 = 0$}
\put (55, 52) {$\phi$}
\end{overpic}
\end{center}
\caption{Proof of \Cref{lem::L_equal_1_local}}
\label{fig::acute_angle}
\end{figure}
\begin{proof}
Let $y$ and $z$ be as in the hypothesis and let $g_i \colon i \in I$ be the active constraints at $z$. Notice that if $y = z \in S$ then $D(y) = G(y) =0$ and there is nothing to prove. So we assume that $y \not\in S$: there exists $i \in \{1,\dots, r \}$ s.t. $g_i(y) < 0$. Moreover, from \Cref{lem::L_equal_1_local_negative} we only need to consider the case where there exists $i \in I$ such that $g_i(y) \ge 0$.

So let $I_+ = I_+(y) =\{ i \in I \colon g_i(y) \ge 0 \}$ and $I_- = I_-(y) =\{ i \in I \colon g_i(y) < 0 \}$. Notice that $I_-$ and $I_+$ depend on $y$, but to obtain a result independent from $I_-$ and $I_+$ it is enough to take the minimum $\epsilon''$ and the maximum $\cst''$ as $I_-$ and $I_+$ vary.

 If we consider the Taylor expansion of $g_i$ at $z$, we obtain:
\begin{equation*}
  i \in I_- \Rightarrow 0 > g_i(y) = \grad g_i(z) \cdot (y-z) + \vb h_i(y) \cdot (y-z),
\end{equation*}
with the same notation as in \Cref{eq::taylor_exp}. This implies that there exists $\delta > 0$ such that $\grad (- g_i)(z) \cdot (y-z) \ge \delta$, for all $i \in I_-$, when $y$ is close enough to $z$.

We want to reduce to the case of only negative inequalities. We define:
\begin{itemize}
  \item $G_-(y) = \abs{\min \{0,g_i(y) \colon i \in I_-\}}$;
  \item $S_- = \cS(g_i \colon i \in I_-)$;
  \item $D_-(y) = \dist{y}{S_-}$;
  \item $T_z S_-$ the (affine) tangent space of $S_-$ at $z$.
\end{itemize}
Notice that, since the gradients are linearly independent, $T_z S_-$ is the affine subspace passing through $z$ and orthogonal to $\grad (- g_i)(z)$ for $i \in I_-$. In particular, since $\grad (- g_i)(z) \cdot (y-z) \ge \delta$, $y-z \notin T_z S_-$ the angle between $y-z$ and $T_z S_-$ is lower bounded by a strictly positive angle $\phi > 0$ for all $y$ close enough to $z$.

For a geometric intuition of the following discussion, see \Cref{fig::acute_angle}. Let $z'$ be the projection of $y$ on $T_z S_-$. By definition of $\phi$ we have $\norm{y - z}_2 \le \frac{\norm{y-z'}_2}{\sin \phi}$. Now let $z''$ be the projection of $y$ on $S_-$. Since $y$ is close to $z$, $z''$ is close to $z'$, i.e. the projection of $y$ on $S_-$ is close to the projection of $y$ on  $T_z S_-$. Thus there exists a constant $c$ such that $\norm{y - z}_2 \le c \frac{\norm{y-z''}_2}{\sin \phi}$.
More precisely, let $z'''$ be the projection of $z''$ on $T_z S_-$.
Thus $z''-z = (z''-z''') + (z'''-z)$, and since we project $z''$ on $T_z S_-$ we have:
\begin{itemize}
  \item $z''-z''' = \sum_{i \in I_-}\gamma_i \grad g_i(z) = \vb J_{I_-}(z)\gamma$ for some $\gamma = (\gamma_i \colon i \in I_-)$;
  \item $z'''-z$ is orthogonal to $\grad g_i(z)$ for $i \in I_-$.
\end{itemize}
 By definition of $z'$ we have $\norm{y-z'}_2 \le \norm{y - z'''}_2 \le \norm{y-z''}_2 + \norm{z'' - z'''}_2$. We show now that $\norm{z'' - z'''}_2$ is small compared to $\norm{y-z}_2$.
Expanding at $z$ for $i \in I_-$ we obtain:
\[
  0 = g_i(z'') = \grad g_i (z) \cdot (z'' - z) + \vb h_i (z'') \cdot (z'' - z) = \grad g_i (z) \cdot (z'' - z''') + \vb h_i (z'') \cdot (z'' - z).
\]
Proceeding as in \Cref{eq::lambda}, we have $\gamma = \vb N_{I_-}(z)^{-1}\vb h (z'') (z''-z)$. Now, since $z''$ is the projection of $y$ on $S_-$ and $z \in S_-$ we have $\norm{z'' - z}_2 \le 2 \norm{y - z}_2$. Thus we deduce:
\begin{align*}
  \norm{y-z'}_2 & \le \norm{y-z''}_2 + \norm{z'' - z'''}_2 \\
                & \le \norm{y-z''}_2 + \norm{\vb J_{I_-}(z)}_2 \norm{\gamma}_2 \\
                & \le \norm{y-z''}_2 + \norm{\vb J_{I_-}(z)}_2 \norm{\vb N_{I_-}(z)^{-1}\vb h (z'') (z''-z)}_2 \\
                & \le \norm{y-z''}_2 +2 \norm{\vb J_{I_-}(z)}_2 \norm{\vb N_{I_-}(z)^{-1}}_2 \norm{\vb h (z'')}_2 \norm{y - z}_2.
\end{align*}
Therefore
\[
  \norm{y - z}_2 \le \frac{\norm{y-z'}_2}{\sin \phi} \le \frac{\norm{y-z''}_2}{\sin \phi} + \frac{2 \norm{\vb J_{I_-}(z)}_2 \norm{\vb N_{I_-}(z)^{-1}}_2 \norm{\vb h (z'')}_2 \norm{y - z}_2}{\sin \phi},
\]
and finally
\[
\big( 1 - \frac{2 \norm{\vb J_{I_-}(z)}_2 \norm{\vb N_{I_-}(z)^{-1}}_2 \norm{\vb h (z'')}_2 }{\sin \phi}\big) \norm{y - z}_2 \le \frac{\norm{y-z''}_2}{\sin \phi}.
\]
As $z'' \to z$ if $y \to z$, $\norm{\vb h (z'')}_2 \to 0$ for $y \to
z$. Then there exists $\epsilon''>0$ such that $D(y) \le \epsilon''$
implies $1 - \frac{2 \norm{\vb J_{I_-}(z)}_2 \norm{\vb N_{I_-}(z)^{-1}}_2 \norm{\vb h (z'')}_2 }{\sin \phi} \ge \frac{1}{2}$
and thus
\begin{equation}
  \label{eq::sin_phi}
  \norm{y - z}_2 \le 2 \frac{\norm{y-z''}_2}{\sin \phi}.
\end{equation}

In other words, we just proved in \Cref{eq::sin_phi} that
\[
  D(y) \le \epsilon'' \Rightarrow D(y) \le \frac{2}{\sin \phi} D_-(y).
\]
Since $D_-$ is the distance function to $S_-$, that is defined by
inequalities negative at $y$, we can apply
\Cref{lem::L_equal_1_local_negative}: there exists $\cst'$ such that
if $\epsilon''$ is small enough, $D_-(y) \le \epsilon''$ implies $D_-(y) \le \cst' G_-(y)$ (notice that this is possible because $D(y) \to 0$ implies $D_-(y) \to 0$). Moreover observe that $G(y) = G_-(y)$ since only the $g_i$ that are negative at $y$ contribute to $G(y)$. Then, if we set $\cst'' = \frac{2\cst'}{\sin \phi}$ we can conclude:
\[
  D(y) \le \epsilon'' \Rightarrow D(y) \le \cst'' G(y).
\]

\end{proof}
We can now show that if the CQC hold at every point of the semialgebraic set the Łojasiewicz exponent is equal to $1$.
\begin{theorem}
  \label{thm::L_equal_1}
  Let $D, G$ as in \Cref{def::loja} and assume that the CQC holds at every point of $S = \cS(\vb g) \subset \ubox$. Then there exists a constant $\cst \in \RR_{>0}$ such that:
  \[
    D(y) \le \cst G(y)
  \]
  for all $y \in \ubox$.
\end{theorem}
\begin{proof}
  Let $\epsilon = \min_{z\in S} \epsilon''(z)$ and $\cst' = \max_{z \in S} \cst''(z)$, with $\epsilon''(z)$ and $\cst''(z)$ as in \Cref{lem::L_equal_1_local}. Notice that $\epsilon>0$ and $\cst' < \infty$ exist because $S$ is compact and $0<\epsilon''(z)$, $\cst''(z)< \infty$
  are respectively lower and upper semicontinuous functions of $z$. Let $U = \{ y \in \ubox \mid \dist{y}{S} <\epsilon \} \subset \ubox$
  (an open set containing $S$): by definition of $\epsilon$ and
  $\cst'$ we have $D(y) \le \cst' G(y)$ for all $y \in U$ from
  \Cref{lem::L_equal_1_local}.

  Now consider the compact set $C = \ubox \setminus U$ and let $G^* >0$ be the minimum of $G$ on $C$. Moreover since $S \subset \ubox$ we have $D(y) \le 2\sqrt{n}$ for $y \in \ubox$. Then:
  \[
    D(y) \le 2\sqrt{n} = \frac{2\sqrt{n}}{G^*}G^* \le \frac{2\sqrt{n}}{G^*} G(y)
  \]
  for all $y \in C$.

  Finally, taking $\cst = \max (\cst', \frac{2\sqrt{n}}{G^*})$ we obtain:
  \[
    D(y) \le \cst G(y)
  \]
  for all $y \in \ubox$.
\end{proof}
\begin{remark}
  In \Cref{thm::L_equal_1} we prove that in regular cases the {\L}ojasiewicz  exponent is 1. On the other hand we don't give a precise estimate for the constant $\cst$, even if we can revisit the proof of \Cref{lem::L_equal_1_local_negative}, \Cref{lem::L_equal_1_local} and \Cref{thm::L_equal_1} to bound it in terms of the following parameters:
  \begin{itemize}
    \item the max norm of the Jacobian of the $\vb g$: we could bound this parameter bounding the norm of $\vb g$;
    \item the min norm of the Gram matrix of the $\grad \vb g$: this measures how close are the gradients to be linearly dependend;
    \item the minimum of $G(y)$ on the complementary in $\ubox$ of a small neighbourhood of $S$: this measures how close are the $\vb g$ to have a common zero outside of $S$;
    \item the convergence rate to $0$ of the Taylor remainder $\vb h (z)$.
  \end{itemize}
  A detailed analysis of these parameters would also give an upper bound for $\cst$, but we don't develop it for the sake of simplicity.
\end{remark}

\begin{remark}
  On the contrary when the problem is not regular the bounds on the exponent $\textit{\L}$ can be large. We have:
  \[
    \textit{\L} \le d(\vb g)(6d(\vb g) - 3)^{n+r-1}
  \]
  see \cite[sec.~3.1]{kurdyka_convexifying_2015} and \cite{kurdyka_metric_2016}.
\end{remark}

\subsection{Construction of a polynomial echelon function}
\label{subsec::echelon}
In this section, we describe the polynomial echelon function $h_{k,m}$
used to perturb $f$.
This echelon polynomial depends on a parameter $\delta\in \R_{>0}$
controlling the width
of the step (and defined in \Cref{subsec::delta}) and on a parameter $k\in
\R_{>0}$ controlling the minimum of the function.
To show that the degree and the norm of the perturbation polynomial depend
polynomially on $\epsilon(f)$ (in \Cref{subsec::bounds}), we are going
to bound the degree of the echelon polynomials in terms of $\delta$
and $k$.

Consider the following function:
\begin{equation}
\label{eq::spline_echelon}
H(t)=
  \begin{cases}
    1 & t \in [-1, -\delta]\\
    -\frac{9(k - 1)}{2 d^3 k}t^3 - \frac{27(k - 1)}{2 d^2 k} t^2 - \frac{27 (k - 1)}{2 d k} t - \frac{7 k - 9}{2 k} & t \in [-\delta, -\delta +\frac{\delta}{3}]\\
    \frac{9(k - 1)}{ d^3 k}t^3 + \frac{27(k - 1)}{2 d^2 k} t^2 + \frac{9 (k - 1)}{2 d k} t + \frac{k + 1}{2 k} & t \in [-\delta +\frac{\delta}{3}, -\delta +\frac{2\delta}{3}]\\
    -\frac{9 (k - 1)}{2 d^3 k}t^3 + \frac{1}{k} & t \in [-\delta +\frac{2\delta}{3}, 0]\\
    \frac{1}{k} & t \in [0,1]
  \end{cases}
\end{equation}
The piecewise polynomial function $H(t)$ is a C$^{2}$ cubic spline on $[-1,1]$. Indeed an explicit computation shows that the functions  $H,H^{(1)},H^{(2)}$ are
absolutely continuous, and moreover the piecewise constant function $H^{(3)}$
is of total variation $V = \frac{216(k-1)}{\delta^3 k}$. Finally notice that $H$ is non-increasing on $[-1, 1]$.


We approximate this function by a polynomial $\in \RRT$, using
Chebyshev approximation (see \Cref{fig::poly_echelon}):
\begin{theorem}[{Chebyshev approximation on $[-1, 1]$ \cite{trefethen_approximation_2013}}]
\label{thm:chebyshev}
  For an integer $u$, let $h: [-1,1]\rightarrow \R$ be a function such that its derivatives through
  $h^{(u-1)}$ be absolutely continuous on $[-1, 1]$ and its $u$-th derivative $h^{(u)}$ is of bounded variation $V$. Then its Chebyshev approximation $p_m$ of degree $m$ satisfies:
\begin{equation*}
    \norm{h - p_m} \le \frac{4V}{\pi u (m - u)^u}.
\end{equation*}
\end{theorem}

\begin{proposition}
\label{prop::echelon}
There exists a univariate polynomial $h_{k,m} \in \RRT$ such that:
\begin{itemize}
  \item $\deg h_{k,m} = m$ with $m=\big\lceil\frac{6}{\delta}\sqrt[3]{\frac{4(k-1)}{3\pi}}+3\big\rceil$;
  \item for $t \in [-1, -\delta]$ we have $1 - \frac{1}{k} \le h_{k,m}(t) \le 1 + \frac{1}{k}$;
  \item for $t \in [0, 1]$ we have $h_{k,m}(t) \le \frac{2}{k}$;
  \item for $t \in [-1, 1]$ we have $0 \le h_{k,m}(t) \le 1 + \frac{1}{k}$.
  \end{itemize}
\end{proposition}
\begin{proof}
We construct a degree $m$ Chebyshev approximation $h_{k,m}\in \RRT$  of $H$ such that
\begin{equation}
\label{eq::poly_echelon}
  \norm{H - h_{k,m}} \le {1\over k},
\end{equation}

so that the last three points of the proposition are satisfied.
As $H$, $H^{(1)}$ and $H^{(2)}$ are absolutely continuous and $H^{(3)}$ has total
variation $V= \frac{216(k-1)}{\delta^3 k}$, by \Cref{thm:chebyshev},
it suffices to take $m$ such $\frac{4V}{3\pi (m-3)^3} \le \frac{1}{k}$, i.e.
\[
  m \ge \sqrt[3]{\frac{4Vk}{3\pi}}+3 = \frac{6}{\delta}\sqrt[3]{\frac{4(k-1)}{3\pi}}+3,
\]
which proves the first point.

The other points follow from \Cref{eq::poly_echelon} and the definition of $H$ in \eqref{eq::spline_echelon}.
\end{proof}

\begin{figure}
\centerline{\includegraphics[width=7cm]{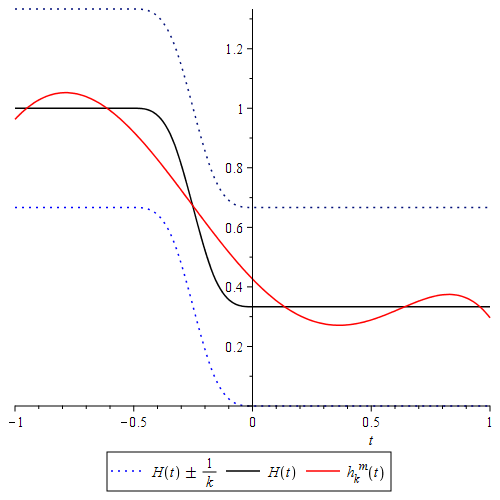}}
  \caption{Polynomial Echelon Function}
  \label{fig::poly_echelon}
\end{figure}
\section{Effective Putinar's Positivstellensatz}
\label{sec::polynomial_putinar}
This section is devoted to the proof of \Cref{thm::polynomial_putinar}.
\subsection{From S to $\ubox$}
\label{subsec::bounds}
Let $h_{k,m}$ be as in \Cref{prop::echelon}. We want to show that, for a suitable choice of $k$, $m$ and $s \in \RR_{>0}$, the polynomial:
\begin{equation}
  \label{eq::perturbation_poly}
  p = f - s \sum_{i=1}^r h_{k,m}(g_i)g_i
\end{equation}
is such that $p \ge \frac{f^*}{2}$ on $\ubox$.
\begin{remark}
Our construction of the perturbed polynomial $p$ is similar to the one
in \cite{schweighofer_optimization_2005}, \cite{nie_complexity_2007}, or \cite{averkov_constructive_2013}
where the polynomial $h$ is a univariate (sum of) squares. That choice is simpler, but it results in worst
  bounds for the degree and the norm of $s \sum_{i=1}^r
  h_{k,m}(g_i)g_i$, than the one we obtain using the polynomial echelon
  function $h_{k,m}$.

These univariate SoS coefficients have also been used in
  \cite{kurdyka_convexifying_2015}, to prove that one can uniformly
  approximate
  positive polynomials on compact sets, using the proper subcone of the quadratic module $\cQ(\vb g)$ where the SoS coefficient of $g_i$ is of the form $\sum_j (h_j(g_i))^2$, for $h_j$ univariate. They derive a Putinar's Positivstellensatz and apply it to Polynomial Optimization problems. We describe the equivalent of the uniform approximation result in \Cref{thm::quad_module_weirstrass}, with our coefficients $h_{k,m} \in \cQ(1+T, 1-T)$.
\end{remark}

\begin{proposition}
\label{prop::bounds}
  Assume that the normalisation assumptions \eqref{assum::norm} are satisfied. If
  \begin{equation}
  \label{eq::s}
    s > \frac{6 \norm{f}}{\delta};
  \end{equation}
  \begin{equation}
  \label{eq::k_1}
    k > \frac{2r-2}{ \delta}+1;
  \end{equation}
  \begin{equation}
  \label{eq::k_2}
    k > \frac{4rs}{f^*};
  \end{equation}
  then $p = f - s \sum_{i=1}^r h_{k,m}(g_i)g_i \ge \frac{f^*}{2}$ on $\ubox$.
\end{proposition}
\begin{proof}
Let $x \in A$ so that $G(x) \ge \delta$, i.e. $\min\{g_{1}(x),\ldots,g_{r}(x),0\}\le -\delta$ (see \Cref{sec::ingredients}), and WLOG assume $g_1(x) \le -\delta$. Notice that from \Cref{prop::echelon} we have $h_{k,m}(g_1(x)) \ge 1 - \frac{1}{k}$
and, if $g_i(x) \ge 0$, $h_{k,m}(g_1(x)) \le \frac{2}{k}$. Moreover recall that $\norm{g_i} \le \frac{1}{2}$ from the normalisation assumptions \eqref{assum::norm}. Then:
\begin{align*}
  p(x) &= f(x) - s \sum_{i=1}^r h_{k,m}(g_i(x))g_i(x)  \\
       &\ge f(x) + s \delta (1 - \frac{1}{k}) - s \sum_{i=2}^r h_{k,m}(g_i(x))g_i(x) \\
       &\ge f(x) + s \delta (1 - \frac{1}{k}) - s \frac{r-1}{k}
        = f(x) + s \frac{\delta}{2} (1 - \frac{1}{k}) + s (\frac{\delta}{2} (1 - \frac{1}{k}) - \frac{r-1}{k}).
\end{align*}
From \Cref{eq::s} and \Cref{eq::k_1}, we have respectively
$f(x) + s \frac{\delta}{2} (1 - \frac{1}{k}) > \frac{\norm{f}}{2} \ge \frac{f^*}{2}$
and $\frac{\delta}{2} (1 - \frac{1}{k}) - \frac{r-1}{k} > 0$, so that
$p(x)> \frac{f^*}{2}$ for $x \in A$.

By \Cref{eq::k_2},  $\frac{3f^*}{4} - \frac{sr}{k} > \frac{f^*}{2}$.
By the normalization assumptions \eqref{assum::norm} and as $h_{k,m}$ is upper bounded by
$2\over k$ on $[0,1]$ (see \Cref{prop::echelon}), we therefore deduce that for $x \in \ubox \setminus A$
\begin{align*}
  p(x) &= f(x) - s \sum_{i=1}^r h_{k,m}(g_i(x))g_i(x)
         \ge \frac{3f^*}{4} - sr \frac{2}{k} \frac{1}{2}= \frac{3f^*}{4}
         - \frac{sr}{k}> \frac{f^*}{2}.
\end{align*}
This shows that $p(x) > \frac{f^*}{2}$ for $x \in \ubox = A \cup (\ubox \setminus A)$.
\end{proof}
\begin{proposition}
Let $p$ be as in \eqref{eq::perturbation_poly}, with \eqref{eq::s}, \eqref{eq::k_1}, \eqref{eq::k_2} and the normalization assumptions \eqref{assum::norm} satisfied. Let  $d(\vb g) = \max_i \deg g_i$. Then
  \begin{equation}
  \label{eq::norm}
    \norm{p} = O(\norm{f} 2^{3\textit{\L}} r \cst d(f)^{2\textit{\L}}\epsilon(f)^{-\textit{\L}}),
  \end{equation}
  \begin{equation}
  \label{eq::degree}
    \deg{p} = O(2^{4\textit{\L}} r^{\frac{1}{3}}\cst^{\frac{4}{3}} d(\vb g) d(f)^{\frac{8\textit{\L}}{3}}\epsilon(f)^{-\frac{4\textit{\L}+1}{3}}).
  \end{equation}
\end{proposition}
\begin{proof}
Let $d = d(f) = \deg f$. We start bounding $m$ in terms of $\epsilon(f)$.

We can choose $m=\big\lceil\frac{6}{\delta}\sqrt[3]{\frac{4(k-1)}{3\pi}}+3\big\rceil$
 from \Cref{prop::echelon}, thus it is enough to bound $k$ and $\delta$.

From \Cref{lem::delta} we can choose $\delta = \frac{1}{\cst}(\frac{\epsilon(f)}{8d^2})^{\textit{\L}} = \cst^{-1}\epsilon(f)^{\textit{\L}} 2^{-3\textit{\L}}d^{-2\textit{\L}}$ . From \Cref{eq::s} we deduce that:
\begin{equation}
  \label{eq::s_order}
  s = O(\frac{\norm{f}}{\delta}) = O(\norm{f} \cst 2^{3\textit{\L}}d^{2\textit{\L}}\epsilon(f)^{-\textit{\L}} ).
\end{equation}
 From \Cref{eq::k_1} we deduce that $k = O(\frac{r}{\delta})$, while from \Cref{eq::k_2} (together with \Cref{eq::s}) we deduce that $k = O(\frac{r}{\epsilon(f)\delta})$:
the latter has an higher order in terms of $\epsilon(f)$, and finally:
\begin{equation}
  \label{eq::k_order}
  k=O(\cst 2^{3\textit{\L}} r d^{2\textit{\L}} \epsilon(f)^{-(\textit{\L}+1)}).
\end{equation}
Now we plug \Cref{eq::k_order} in $m=\big\lceil\frac{6}{\delta}\sqrt[3]{\frac{4(k-1)}{3\pi}}+3\big\rceil$ and obtain:
\begin{equation}
  \label{eq::m_degree}
  m = O(\frac{k^{\frac{1}{3}}}{\delta}) = O((\cst^{\frac{1}{3}}r^{\frac{1}{3}}2^{\textit{\L}} d^{\frac{2{\textit{\L}}}{3}}\epsilon(f)^{-\frac{{\textit{\L}}+1}{3}})(\cst 2^{3{\textit{\L}}}d^{2{\textit{\L}}}\epsilon(f)^{-{\textit{\L}}})) = O (\cst^{\frac{4}{3}}r^{\frac{1}{3}}2^{4{\textit{\L}}}d^{\frac{8{\textit{\L}}}{3}}\epsilon(f)^{-\frac{4{\textit{\L}}+1}{3}}).
\end{equation}

By the normalization assumptions \eqref{assum::norm}, the properties
of $h_{k,m}$ (\Cref{prop::echelon}) and \Cref{eq::s_order} we obtain:
\begin{align*}
\norm{p} &\le \norm{f} + s \sum_{i=1}^r
\norm{h_{k,m}(g_i)g_i} \le \norm{f} + sr (1+ {1\over
                   k}){1 \over 2} \\
  &\le \norm{f} + sr  = O(\norm{f} + \norm{f} \cst r 2^{3{\textit{\L}}}d^{2{\textit{\L}}}\epsilon(f)^{-{\textit{\L}}})  \\
  & = O(\norm{f}\cst r 2^{3{\textit{\L}}}d^{2{\textit{\L}}}\epsilon(f)^{-{\textit{\L}}}).
\end{align*}

Similarly, using \Cref{eq::m_degree} we have:
\[
\deg(f-p)\le \max_i \{ \deg (h_{k,m}(g_{i}) g_{i}), i=1, \ldots, r\}
= O(d(\vb g) m + d(\vb g)) = O(2^{4{\textit{\L}}} r^{\frac{1}{3}}
\cst^{\frac{4}{3}} d(\vb g)d^{\frac{8{\textit{\L}}}{3}}\epsilon(f)^{-\frac{4{\textit{\L}}+1}{3}}),
\]
where $d(\vb g) = \max_i \deg g_i$.
\end{proof}

We now show that $f-p= s \sum_{i=1}^r h_{k,m}(g_i(x))g_i(x)$ is in
$Q_\ell(\vb g)$, giving degree bounds for the degree $\ell$ that is necessary to represent $f - p$ (see \Cref{prop::perturbed_poly}).
\begin{theorem}[Fekete -  Lukács , \cite{powers_polynomials_2000}]
\label{thm::markov-felete-lucas}
  Let $f \in \RRT_d$ be a univariate polynomial of degree $d$. If $f \ge 0 $ on $[-1, 1]$ then there exists $s_0, s_1, s_2 \in \Sigma^2$ such that $f = s_0 + s_1 (1-T) + s_2(1+T)$, where the degree of every addendum is $\le d+1$.
In other words, $\pos([-1, 1])_d \subset \cQ_{d+1}(1-T, 1+T)$.
\end{theorem}
\begin{proof}
  From \cite{powers_polynomials_2000} (see also \cite[part~VI, 46--47]{polya_problems_1976}) there exists polynomials $h_i$ such that $f = h_0^2 + h_1^2(1-T) + h_2^2 (1+T) + h_3^2 (1-T^2)$, where the degree of every addendum is $\le d$. Now notice that $1-T^2 = \frac{1}{2}\big( (1+T)^2(1-T) + (1-T)^2(1+T)\big)$ to conclude.
\end{proof}
Our assumption is that $\cQ(1-\norm{\vb X}_2^{2}) \subset \cQ(\vb g)$, while we are trying to reduce to the case of $\ubox$. We show that we can move from the latter to the former with a constant  degree shift in \Cref{lem::box_to_ball}.
  \begin{lemma}
  \label{lem::box_to_ball}
    The preordering associated with the box $\ubox$ is included in the quadratic module of the unit ball. In particular $\cO_{d}(1 \pm X_i \colon i \in \{1, \dots ,n\}) \subset \cQ_{d+n} (1 - \norm{\vb X}_2^2)$.
  \end{lemma}
  \begin{proof}
  Notice that:
   \[ 1 \pm X_i = \frac{1}{2} ((1-X_i^2 + (1 \pm X_i)^2)) = \frac{1}{2} ((1-\norm{\vb X}_2^2+ \sum_{j \neq i} X_j^2 + (1 \pm X_i)^2)).\] This implies that $\cQ_d(1 \pm X_i \colon i \in \{1, \dots ,n\}) \subset \cQ_{d+1} (1 - \norm{\vb X}^2)$.
  Since $\cQ (1 - \norm{\vb X}^2)$ is a preordering (i.e. it is closed under multiplication) we also have $\cO_d(1 \pm X_i \colon i \in \{1, \dots ,n\}) \subset \cQ_{d+n} (1 - \norm{\vb X}_2^2)$.
  \end{proof}
\Cref{lem::box_to_ball} implies that we have a Putinar-like representation of polynomials strictly positive on the box as elements of the quadratic module of the ball.
\begin{lemma}
\label{lem::putinar_ball}
  Let $\cQ(\vb g)$ be a quadratic module such that $1-\norm{\vb X}_2^2 \in \cQ(\vb g)$, and let $f$ be a polynomial such that $f > 0$ on $\ubox$. Then $f \in \cQ(\vb g)$.
\end{lemma}
\begin{proof}
  Since $f > 0$ on $\ubox$, then $f \in
  \cQ(1-\norm{\vb X}^2)$ by Schm\"udgen's Positivstellensatz and \Cref{lem::box_to_ball}.
  Now by hypothesis $\cQ(1 - \norm{\vb X}^2) \subset \cQ(\vb g)$ and thus $f \in \cQ(\vb g)$.
\end{proof}
\Cref{lem::putinar_ball} shows that we can use a Schmüdgen theorem on $\ubox$, for instance \Cref{ass::putinar_on_B}, to prove that $f \in \cQ(\vb g)$, without having proved a general Putinar's Positivstellensatz for $\cQ(\vb g)$ yet. Another alternative to prove the result would have been to notice that $f > 0$ on $\ubox$ implies $f > 0$ on the unit ball, and then apply a Schmüdgen/Putinar theorem for $\cQ(1-\norm{\vb X}^2)$.

We are ready to show that the addenda $h(g_i) g_i$ belong to $\cQ(\vb g)$, with degree bounds for the representation.
\begin{lemma}
\label{lem::inclusion_quad_module}
  Let $h \in \pos([-1, 1])_m$ be a univariate polynomial of degree
  $m$. If the normalization assumptions \eqref{assum::norm} are satisfied and $d(\vb g) = \max_i \deg g_i$, then  $h(g_i) g_i \in \cQ_{d(\vb g) m +
    \ell_0+2}(\vb g)$, where $\ell_0 = \min \{ k \colon 1 - g_i \in
  \cQ_k(\vb g)\ \forall i= 1,\ldots,r \}$.
\end{lemma}
\begin{proof}
  By \Cref{thm::markov-felete-lucas}, $h \in \cQ_{m+1}(1+T, 1-T)$, i.e. $h = s_0 + s_1 (1+T) + s_2 (1-T)$, where $s_i$ is a SoS where $\deg s_0$, $\deg s_1+1$ and $\deg s_2 + 1$ are $\le m+1$. Let $d_i = \deg g_i$.
  Notice that:
  \begin{itemize}
    \item $s_0(g_i) g_i \in \cQ_{d_i (m+1) + d_i}(\vb g) = \cQ_{d_i (m+2)}(\vb g)$ since $s_0$ is a SoS of degree $\le m+1$;
    \item $s_1(g_i) (1+g_i)g_i = s_1(g_i)g_i+s_1(g_i) g_i^2 \in
      \cQ_{d_i m+ 2 \,d_i}(\vb g)$ since $s_1$ is a SoS of degree $\le m$;
    \item $s_2(g_i) (1-g_i)g_i = s_2(g_i)(g_i-g_i^2) \in \cQ(\vb g)$. Indeed $g_i-g_i^2 = (1 - g_i)^2 g_i + g_i^2 (1 - g_i)$, and since $\norm{g_i} \le \frac{1}{2}$ we have
    $(1 - g_i) \in \cQ(\vb g)$ by \Cref{lem::putinar_ball}. In particular let $\ell_0$ be minimal such that for all $i$ we have $1 - g_i \in \cQ_{\ell_0}(\vb g)$. Then $g_i - g_i^2 \in Q_{\ell_0+2}(\vb g)$ and finally $s_2(g_i)(g_i-g_i^2) \in \cQ_{d_i m+\ell_0+2}(\vb g)$.
  \end{itemize}
  This shows that $h(g_i)g_i = s_0(g_i) g_i + s_1(g_i) (1+g_i)g_i +
  s_2(g_i) (1-g_i)g_i \in \cQ_{d(\vb g) m + \ell_0 +2}(\vb g)$, where $d(\vb g) = \max_i d_i$.
\end{proof}
We now apply \Cref{lem::inclusion_quad_module} to $p$ to determine the degree of the representation of $f-p \in \cQ(\vb g)$.
\begin{proposition}
  \label{prop::perturbed_poly}
  Let $s \sum_{i=1}^r h_{k,m}(g_i)g_i = f - p$ be as in
  \eqref{eq::perturbation_poly}. If the normalization assumptions \eqref{assum::norm} are satisfied, then $f - p \in \cQ_{\ell}(\vb g)$ when
  $\ell = O(2^{4\textit{\L}} r^{\frac{1}{3}} \cst^{\frac{4}{3}} d(\vb g) d(f)^{\frac{8\textit{\L}}{3}}\epsilon(f)^{-\frac{4\textit{\L}+1}{3}})$,
  where $\cst, \textit{\L}$ are given by \Cref{lem::delta}.
\end{proposition}
\begin{proof}
  It is enough to prove that for all $i$ we have $h_{k,m}(g_i)g_i \in
  \cQ_{\ell}(\vb g)$. Notice that $h_{k,m}(g_i)g_i \in \cQ_{d(\vb g) m +
    \ell_0 + 2}(\vb g)$ for all $i$, see
  \Cref{lem::inclusion_quad_module}. From \Cref{eq::m_degree} we can
  choose $m =
  O(\cst^{\frac{4}{3}}r^{\frac{1}{3}}2^{4\textit{\L}}d^{\frac{8\textit{\L}}{3}}\epsilon(f)^{-\frac{4\textit{\L}+1}{3}})$
  and thus if $\ell = O(2^{4\textit{\L}} r^{\frac{1}{3}} \cst^{\frac{4}{3}}  d(\vb g)\,d(f)^{\frac{8\textit{\L}}{3}}\epsilon(f)^{-\frac{4\textit{\L}+1}{3}})$ we have $s \sum_{i=1}^r h_{k,m}(g_i)g_i = f - p \in \cQ_{\ell}(\vb g)$.
\end{proof}

\subsection{The Polynomial Effective Positivstellensatz}
We will use an effective version of Schm\"udgen's Positivstellensatz for the box $\ubox$.
  \begin{theorem}[{\cite{laurent_effective_2021}}]
    \label{ass::putinar_on_B}
      Let $f \in \RRg$, $\deg f = d$ and $f > 0$ on $\ubox$. Let $f_{\min} = \min_{x \in \ubox} f(x)$ and $f_{\max} = \max_{x \in \ubox} f(x)$. Then there exists a constant $C(n, d)$ (depending only on $n$ and $d$) such that $f \in \cO_{nr}(1 \pm X_i \colon i \in \{1, \dots ,n\})$, where:
      \[
        r \ge \max \left\{ \pi d \sqrt{2n}, \sqrt{\frac{C(n, d)(f_{\max} - f_{\min})}{f_{\min}}} \right\}.
      \]
      Moreover the constant $C(n,d)$ is a polynomial in $d$ for fixed $n$:
      \[
        C(n,d) \le 2\pi^2 d^2(d+1)^n n^3 = O(d^{n+2}n^3)
      \]
  \end{theorem}
Our assumption is that $\cQ(1-\norm{\vb X}_2^{2}) \subset \cQ(\vb g)$, while \Cref{ass::putinar_on_B} involves $\cO(1 \pm X_i \colon i \in \{1, \dots ,n\})$. But we have already shown in \Cref{lem::box_to_ball} that we can move from the latter to the former with a constant  degree shift.

We are now ready to prove the main theorem.
  \begin{proof}[{Proof of \Cref{thm::polynomial_putinar}}]
    Let $p = f - s \sum_{i=1}^r h_{k,m}(g_i)g_i$ be as in \Cref{eq::perturbation_poly}, with $s, k, m$ satisfying $\Cref{eq::s}$, $\Cref{eq::k_1}$, $\Cref{eq::k_2}$ and $h_{k,m}$ as in \Cref{prop::echelon}. In particular:
    \begin{itemize}
      \item $p \ge \frac{f^*}{2}$ on $\ubox$ from \Cref{prop::bounds};
      \item $\norm{p} =
        O(
        2^{3\textit{\L}}
        r\,
        \cst\,
        d(f)^{2\textit{\L}}
\norm{f}
        \epsilon(f)^{-\textit{\L}})$ from \Cref{eq::norm};
      \item $\deg p = O(
        2^{4\textit{\L}}
        r^{\frac{1}{3}}
        \cst^{\frac{4}{3}}
        d(\vb g)\,
        d(f)^{\frac{8\textit{\L}}{3}}
        \epsilon(f)^{-\frac{4\textit{\L}+1}{3}})$ from \Cref{eq::degree}.
    \end{itemize}
We apply \Cref{ass::putinar_on_B} to $p$:
    $p \in \cO_{n\ell_0}(1 \pm X_i \colon i \in \{1, \dots ,n\})$, if $\ell_0 \ge \sqrt{\frac{C(n, \deg p)(p_{\max} - p_{\min})}{p_{\min}}}$. Recall also from \Cref{ass::putinar_on_B} that $C(n, m) = O(n^3 m^{n+2})$.
    We now deduce the asymptotic order of $\ell_0$:
    \begin{align*}
          \sqrt{\frac{C(n, \deg p)(p_{\max} - p_{\min})}{p_{\min}}}
        = &  O\big(\sqrt{n^3 (\deg p)^{n+2}(\frac{2\norm{p}}{f^*}+1)}\big) \\
      = & O\big(\sqrt{
          n^3 (
          2^{4\textit{\L}}
          r^{\frac{1}{3}}
          \cst^{\frac{4}{3}}
          d(\vb g)\,
          d(f)^{\frac{8\textit{\L}}{3}}
          \epsilon(f)^{-\frac{4\textit{\L}+1}{3}})^{n+2}\frac{\norm{f}
          2^{3\textit{\L}}
          r
          \cst
          d(f)^{2\textit{\L}}\epsilon(f)^{-\textit{\L}}}{f^*}} \big)\\
      = & O\big((
          n^3
          2^{(4n+11)\textit{\L}}
          r^{\frac{n+5}{3}}
          \cst^{\frac{4n+11}{3}}
          d(\vb g)^{n+2}
          d(f)^{\frac{2(4n+11)\textit{\L}}{3}}
          \epsilon(f)^{-\frac{(4\textit{\L}+1)n+11\textit{\L}+5}{3}})^{\frac{1}{2}}\big) \\
      = & O\big(
          n^{\frac{3}{2}}
          2^{\frac{(4n+11)\textit{\L}}{2}}
          r^{\frac{n+5}{6}}
          \cst^{\frac{4n+11}{6}}
          d(\vb g)^{\frac{n+2}{2}}
          d(f)^{\frac{(4n+11)\textit{\L}}{3}}
          \epsilon(f)^{-\frac{(4\textit{\L}+1)n+11\textit{\L}+5}{6}}\big),
    \end{align*}
    so we can choose
    $\ell_0 = O(n^{\frac{3}{2}}
    2^{\frac{(4n+11)\textit{\L}}{2}}
    r^{\frac{n+5}{6}}
    \cst^{\frac{4n+11}{6}}
    d(\vb g)^{\frac{n+2}{2}}
    d(f)^{\frac{(4n+11)\textit{\L}}{3}}
    \epsilon(f)^{-\frac{(4\textit{\L}+1)n+11\textit{\L}+5}{6}})$ and $p \in \cO_{n\ell_0}(1 \pm X_i \colon i \in \{1, \dots ,n\})$.
     Now, from \Cref{lem::box_to_ball} we have $\cO_{n\ell_0}(1 \pm X_i \colon i \in \{1, \dots ,n\}\subset \cQ_{n\ell_0+n}(1-\norm{\vb X}_2^2)$. Moreover from \Cref{assum::norm} we have that $1 - \norm{\vb X}_2^2 \in \cQ(\vb g)$.
     In particular if $1 - \norm{\vb X}_2^2 \in \cQ_{\ell_1}(\vb g)$ and
     thus $\cQ_{n\ell_0+n}(1-\norm{\vb X}_2^2)\subset\cQ_{n\ell_0 + n + \ell_1} (\vb
     g)$, i.e. choosing
$\ell = nO(\ell_0) =
     O(n^{\frac{5}{2}}
2^{\frac{(4n+11)\textit{\L}}{2}}
r^{\frac{n+5}{6}}
\cst^{\frac{4n+11}{6}}
d(\vb g)^{\frac{n+2}{2}}
d(f)^{\frac{(4n+11)\textit{\L}}{3}}
\epsilon(f)^{-\frac{(4\textit{\L}+1)n+11\textit{\L}+5}{6}})$
     we have $p \in \cQ_\ell(\vb g)$.
     Finally notice that $f = (f - p) + p$ and
     \begin{itemize}
       \item $p \in \cQ_\ell(\vb g)$ from the discussion above;
       \item $f-p \in \cQ_\ell(\vb g)$ from \Cref{prop::perturbed_poly}, since the degree of the truncated quadratic module in \Cref{prop::perturbed_poly} is smaller than $\ell$.
     \end{itemize}
     Then $f \in \cQ_\ell(\vb g)$ with
\begin{equation} \label{eq::sharp exp}
     \ell = O(n^{\frac{5}{2}}
     2^{\frac{(4n+11)\textit{\L}}{2}}
     r^{\frac{n+5}{6}}
     \cst^{\frac{4n+11}{6}}
     d(\vb g)^{\frac{n+2}{2}}
     d(f)^{\frac{(4n+11)\textit{\L}}{3}} \epsilon(f)^{-\frac{(4\textit{\L}+1)n+11\textit{\L}+5}{6}}).
\end{equation}

We simplify the exponents for readibility. Recall that $\textit{\L}\ge 1$ and $\cst \ge 1$, and assume $n \ge 2$. Under these assumptions the inequalities $(4n+11)\textit{\L} \le 10n\textit{\L}$, $n+5 \le 6n$, $4n + 11 \le 10n$, $n+2 \le 2n$ and $(4\textit{\L}+1)n+11\textit{\L}+5 \le 14n\textit{\L}$ hold. Therefore we deduce that  $f \in \cQ_\ell(\vb g)$
if
\begin{align*}
\ell & \ge  O(
n^{3}
2^{5n\textit{\L}}
r^{n}
\cst^{2n}
d(\vb g)^{n}
d(f)^{3.5n\textit{\L}}
\epsilon(f)^{-{2.5 n \textit{\L}}})\\
& = \gamma(n, \vb g) d(f)^{3.5n\textit{\L}}
\epsilon(f)^{-{2.5 n \textit{\L}}},
\end{align*}
where
$\gamma(n,\vb g) =
O(
n^{3}
2^{5n\textit{\L}}
r^{n}
\cst^{2n}
d(\vb g)^{n})\ge 1$.
\end{proof}
\begin{remark}
From \Cref{eq::sharp exp}, we
have
$\ell = O(n^{\frac{5}{2}}
2^{\frac{(4n+11)\textit{\L}}{2}}
r^{\frac{n+5}{6}}
\cst^{\frac{4n+11}{6}}
d(\vb g)^{\frac{n+2}{2}}
d(f)^{\frac{(4n+11)\textit{\L}}{3}} \epsilon(f)^{-\frac{(4\textit{\L}+1)n+11\textit{\L}+5}{6}})$, where $\cst, \textit{\L}$ are defined in \Cref{def::loja}.
The exponents in \Cref{thm::polynomial_putinar} have been simplified for the
sake of readability and are not optimal.
\end{remark}

If the inequalities defining $S$ satisfy a regularity condition we can simplify the bound, since $\textit{\L} = 1$ in this case (see \Cref{subsec::delta}).

\begin{corollary}
  \label{cor::polynomial_putinar_regular}
Assume $n\ge 2$ and let $g_{1},\ldots, g_{r}\in \R[\vb X] = \R[X_1,\dots,X_n]$ satisfying the normalization
assumptions \eqref{assum::norm} and such that the CQC (\Cref{def::CQC}) hold at every point of $\cS(\vb g)$. Let $f\in \R[\vb X]$ such that
 $f^* = \min_{x \in S} f(x) > 0$. Then $f \in \cQ_\ell(\vb g)$ if
  \[
    \ell =
O(n^{3}
2^{5n}
r^{n}
\cst^{2n}
d(\vb g)^{n}
d(f)^{3.5n}
\epsilon(f)^{-{2.5 n}}),
  \]
  where $\cst$ is given by \Cref{thm::L_equal_1}.
\end{corollary}
\begin{proof}
  Apply \Cref{thm::polynomial_putinar} and \Cref{thm::L_equal_1}.
\end{proof}
\section{Convergence of Lasserre's relaxations optimum}
\label{sec::lasserre_hierarchy}
We begin with a short description of Polynomial Optimization Problems (POP) and of the Lasserre hierarchies to approximately solve them, and refer to \cite{lasserre_global_2001}, \cite{lasserre_introduction_2015} for more details.

Let $f, g_1,\dots,g_s \in \RRg$. The goal of Polynomial Optimization is to find:
\begin{equation}\label{eq:pop}
      f^* \coloneqq \inf \ \big\{ \, f(x)\in \R \mid x \in \R^n, \ g_i(x) \ge 0 \ \textup{ for } i=1, \ldots,s \,\big\} = \inf_x f(x) \colon g_i(x) \ge 0 \ \forall i \in \{1,\dots,r\},
    \end{equation}
that is the infimum $f^{*}$
of the \emph{objective function} $f$ on the \emph{basic closed semialgebraic set} $S = \cS(\vb g) $.
It is a general problem, which appears in many contexts and with many applications, see for instance \cite{lasserre_moments_2010}.

We define the \emph{SoS relaxation of order $\ell$} of problem \eqref{eq:pop} as $\tqgen{g}{2\ell}$ and the supremum:
\begin{equation}
    \label{def::sosrel}
  f^*_{\sos,\ell}  \coloneqq \sup \big\{ \, \lambda \in \R \mid f-\lambda  \in \tqgen{g}{2\ell} \,\big\}.
\end{equation}

Now we want to define the dual approximation of the polynomial optimization problem. We are interested in an affine hyperplane section of the cone $\cL_\ell(\vb g) = \cQ_\ell(\vb g)^{\vee}$:
\[
    \cLone_\ell(\vb g) = \big\{ \, L \in \cL_\ell(\vb g) \mid \braket{L}{1}=1 \, \big\}.
\]
With this notation we define the \emph{MoM relaxation of order $\ell$} of problem \eqref{eq:pop} as $\cL_{2\ell}(\vb g)$ and the infimum:
\begin{equation}
\label{def::momrel}
  f^*_{\mom,\ell}  \coloneqq \inf \big\{ \, \braket{L}{f} \in \R \mid L \in \cLone_{2\ell}(\vb g) \,\big\}.
\end{equation}

It is easy to show that the relaxations \eqref{def::sosrel} and \eqref{def::momrel} are lower approximations of $f^*$. Their convergenge to $f^*$ as the order $\ell$ goes to infinity is deduced from Putinar's Positivstellensatz. In particular the rate of convergence can be deduced from the Effective Putinar's Positivstellensatz: see \Cref{thm::polynomial_lasserre}. The proof of this result is the purpose of \Cref{sec::lasserre_hierarchy}.

\begin{remark}
  We have that $f^*_{\sos, \ell} \le f^*_{\mom, \ell}\le f^*$ for all $\ell$. Thus the results of this section, stated for the SoS relaxations $f^*_{\sos, \ell}$, are also valid for the MoM relaxations $f^*_{\mom, \ell}$.
\end{remark}

A first step for the proof of \Cref{thm::polynomial_lasserre} is to recognise \Cref{thm::polynomial_putinar} as a quantitative result of approximation of polynomials with polynomials in the truncated quadratic module.

\begin{theorem}
  \label{thm::quad_module_weirstrass}
  Assume $n\ge 2$ and let $\vb g$ satisfy the normalization conditions \eqref{assum::norm}.
Let $\textit{\L}$ be the {\L}ojasiewicz exponent defined in \Cref{def::loja} and
  let $f \ge 0 $ on $\cS(\vb g)$. Then for $0< \epsilon \le \norm{f}$, we
  have $f -f^{*} +\epsilon = q \in \cQ_\ell(\vb g)$  for
  \begin{equation}
  \label{eq::f_eta_bound}
    \ell \ge \gamma'(n,\vb g)\,  d(f)^{3.5n\textit{\L}}\, \norm{f}^{2.5n\textit{\L}} \epsilon^{-2.5n\textit{\L}}
  \end{equation}
  where
  $\gamma'(n,\vb g) = 3^{2.5n\textit{\L}} \gamma(n, \vb g)\ge 1$ depends only on $n$ and $\vb g$ and $\gamma(n, \vb g)$ is given by \Cref{thm::polynomial_putinar}.
\end{theorem}
\begin{proof}
Notice that $f-f^{*}+\epsilon > 0$ on $\cS(\vb g)$ and
\[
\epsilon(f-f^{*}+\epsilon) =
\frac{\epsilon}{\norm{f-f^{*}+\epsilon}}
\ge \frac{\epsilon}{\norm{f}+|f^{*}|+\epsilon}
\ge \frac{\epsilon}{3 \norm{f}}
\]
for $\epsilon\le \norm{f}$.
Moreover $\deg f-f^{*}+\epsilon = \deg f = d(f)$.
By \Cref{thm::polynomial_putinar}, we have $f-f^{*}+\epsilon = q \in
Q_{\ell}(\vb g)$ if
\begin{align*}
  \ell &\ge O(
      n^{3}
2^{5n\textit{\L}}
r^{n}
\cst^{2n}
d(\vb g)^{n}
d(f)^{3.5n\textit{\L}}
    (\frac{\epsilon}{3 \norm{f}})^{-2.5n\textit{\L}})\\
  &  = \gamma'(n,\vb g)\,  d(f)^{3.5n\textit{\L}}\, \norm{f}^{2.5n\textit{\L}} \epsilon^{-2.5n\textit{\L}}
\end{align*}
where $\gamma'(n,\vb g) = 3^{2.5n\textit{\L}}\gamma(n, \vb g) =
O(
n^{3}
2^{5n\textit{\L}}
3^{2.5n\textit{\L}}
r^{n}
\cst^{2n}
d(\vb g)^{n})\ge 1$ depends only on $n$ and $\vb g$, and not
    on $f$, and $\gamma(n, \vb g)$ is given by \Cref{thm::polynomial_putinar}.
\end{proof}

\begin{remark}
From \Cref{eq::sharp exp}, we
have
$\gamma(n,\vb g)=O(
  n^{\frac{3}{2}}
  2^{\frac{4{\textit{\L}}n+11{\textit{\L}}}{2}}
  r^{\frac{n+5}{6}}
  \cst^{\frac{4n+11}{6}}
  d(\vb g)^{\frac{n+2}{2}}
  )$, where $\cst, {\textit{\L}}$ are defined in \Cref{def::loja}.
The exponents of $\gamma'(n, \vb g) = 3^{2.5n{\textit{\L}}}\gamma(n, \vb g)$ in the proof have been simplified for the
sake of readability and are not optimal.
\end{remark}

\begin{remark}
\Cref{thm::quad_module_weirstrass} is a quantitive version of Weierstrass approximation theorem
for positive polynomials on $S$, showing that a polynomial $f\in
\pos(S(\vb g))$ can be approximated uniformly on $\ubox$ (within distance
$\epsilon$) by an element $f^{*}+q \in Q_{\ell}(\vb g)$ for $\ell \ge \gamma'(n,\vb g)\,  d(f)^{3.5n{\textit{\L}}}\, \norm{f}^{2.5n\textit{\L}} \epsilon^{-2.5n\textit{\L}}$.
\end{remark}

We are now ready to prove the rate of convergence for Lasserre hierarchies.

\begin{theorem}\label{thm::polynomial_lasserre_sos}
  With the same hypothesis of \Cref{thm::quad_module_weirstrass},
  let $f^*_{\sos, \ell}$ be the Lasserre
  SoS (lower) approximation. Then $f^* - f^*_{\sos, \ell}\le \epsilon$
  for
  \begin{equation}
  \label{eq::polynomial_lasserre}
    \ell \ge \gamma'(n,\vb g)\,  d(f)^{3.5n\textit{\L}}\, \norm{f}^{2.5n\textit{\L}} \epsilon^{-2.5n\textit{\L}}.
  \end{equation}
\end{theorem}
\begin{proof}
Notice that
  \[
    f^*_{\sos,\ell} = \sup \{ \, \lambda \in \RR \mid f-\lambda \in \cQ_{2\ell}(\vb g) \, \} =\inf \{ \, \epsilon \in \RR_{\ge 0} \mid f-f^*+\epsilon \in \cQ_{2\ell}(\vb g) \, \}.
  \]
%
By \Cref{thm::quad_module_weirstrass}, for $\ell \ge \gamma'(n,\vb g)\,  d(f)^{3.5n\textit{\L}}\, \norm{f}^{2.5n\textit{\L}} \epsilon^{-2.5n\textit{\L}}$,
$f-f^*+\epsilon \in \cQ_{\ell}(\vb g)$.
This implies that $f^* - f^*_{\sos, \ell}\le \epsilon$ and concludes the proof.
\end{proof}
\begin{theorem}
  \label{thm::polynomial_lasserre}
  With the same hypothesis of \Cref{thm::polynomial_lasserre_sos} and $\gamma''(n,\vb g) = \gamma'(n,\vb g)^{\frac{1}{2.5n\textit{\L}}}$, we
  have
  \[
    0 \le f^* - f^*_{\sos,\ell} \le \gamma''(n,\vb g)
    \norm{f}
      d(f)^{7\over 5} \ell^{-\frac{1}{2.5n\textit{\L}}}.
  \]
\end{theorem}
\begin{proof}
We apply \Cref{thm::polynomial_lasserre_sos} with $\epsilon\le \norm{f}$ such that
$\ell=\lceil \gamma'(n,\vb g)
d(f)^{3.5n\textit{\L}}\norm{f}^{2.5n\textit{\L}}\epsilon^{-2.5n\textit{\L}} \rceil$
and $\gamma''(n,\vb g) = \gamma'(n,\vb g)^{\frac{1}{2.5n\textit{\L}}}$.
\end{proof}

In conclusion \Cref{thm::polynomial_putinar} allows to prove
\Cref{thm::polynomial_lasserre}, a polynomial convergence of the
Lasserre's lower approximations to $f^*$. In comparison with
\cite[th.~8]{nie_complexity_2007}, where the convergence is
logarithmic in level $\ell$ of the hierarchy,
\Cref{thm::polynomial_lasserre} gives a polynomial convergence to $f^{*}$.

In regular POP we can simplify the bound, since $\textit{\L} = 1$ in this case (see \Cref{subsec::delta}).

\begin{corollary}
\label{cor::polynomial_lasserre}
With the same hypothesis of \Cref{thm::polynomial_lasserre_sos} and $\gamma''(n,\vb g) = \gamma'(n,\vb g)^{\frac{1}{2.5n}}$, we
have
\[
  0 \le f^* - f^*_{\sos,\ell} \le \gamma''(n,\vb g)
  \norm{f}
    d(f)^{7\over 5} \ell^{-\frac{1}{2.5n}}
\]
if the CQC (\Cref{def::CQC}) hold at every point of $\cS(\vb g)$.
\end{corollary}
\begin{proof}
  Apply \Cref{thm::polynomial_lasserre} and \Cref{thm::L_equal_1}.
\end{proof}
\section{Convergence of pseudo-moment sequences to measures}
\label{sec::convergence_of_moments}
We are interested in the study of the truncated positive linear functionals $\cL_\ell(\vb g) = \cQ_\ell(\vb g)^{\vee}$, i.e. the dual convex cone of the truncated quadratic modules, and in particular of its section $\cLone_d(\vb g)$. This cone is used to define the Lasserre MoM relaxations \eqref{def::momrel}. In the following we often restrict the linear functionals to polynomials of degree $\le t$, that is we consider the cones $\cL_\ell(\vb g)^{[t]}$.

Notice in particular that, if $\mu \in \cM(S)^{[t]}$ and $q \in \cQ_\ell(\vb g)\cap \RRg_t$ then $\braket{\mu}{q} = \int q \dd{\mu} \ge 0$, since $q \ge 0$ on $S$. In other words: $\cM(S)^{[t]} \subset \cL_\ell(\vb g)^{[t]}$ for all $\ell$,
i.e. our dual cone is an outer approximation of the cone of measures
supported on $S$. To compare quantitatively these cones we consider
their affine sections $\cMone(S)^{[t]}$ and $\cLone_\ell(\vb
g)^{[t]}$. Recall that $\cLone_\ell(\vb
g)^{[t]}$ is a generating section of $ \cL_\ell(\vb g)^{[t]}$ when $t \le \frac{\ell}{2}$, see \Cref{subsec::truncated_moments}. In this section, we prove
\Cref{thm::moment_rate_convergence_prob}, which shows the convergence
of the outer approximation as $\ell$ goes to infinity, and deduce the speed
rate from \Cref{thm::polynomial_putinar}.
To measure this convergence we use the Hausdorff distance of sets $\hdist{\cdummy}{\cdummy}$.

Before the proof of the main theorem, recall that in the finite
dimensional vector space $\RRg_t$, all the norms are equivalent:
we specify in \Cref{lem::norm_comparison} a constant that we will need in the proof of
\Cref{thm::moment_rate_convergence}, for the following norms.
For $f = \sum_{\abs{\alpha}\le t} a_{\alpha} \vb X^{\alpha} \in \RRg_t$,
as usual  $\displaystyle \norm{f} = \max_{\vb x \in [-1,1]^{n}} \abs{f(\vb x)}$, and
  $\displaystyle \norm{f}_2 = \sqrt{\sum_{\abs{\alpha}\le t} a_{\alpha}^2}$.

\begin{lemma}
  \label{lem::norm_comparison}
 For $f\in \RRg_{t}$, we have $\norm{f} \le \sqrt{\binom{n+t}{t}} \norm{f}_2$.
\end{lemma}
\begin{proof}
  Let $ x \in [-1,1]^{n}$ such that $\abs{f(x)} = \norm{f}$.
  Denote $\bar{\vb x} = ( x^{\alpha})_{\abs{\alpha}\le t}$ and $\bar{a} = (a_{\alpha})_{\abs{\alpha}\le t}$. Then:
  \[
    \norm{f} = \abs{f(x)} = \abs{\bar{a} \cdot \bar{x}} \le \norm{\bar a}_2 \norm{\bar{x}}_2 = \norm{f}_2 \norm{\bar{x}}_2
  \]
  using the Cauchy-Schwarz inequality. Finally notice that $\abs{x^{\alpha}} \le 1$ for all $\alpha$ since $x \in [-1,1]^{n}$, and thus $\norm{\bar{x}}_2 \le \sqrt{\dim \RRg_t} = \sqrt{\binom{n+t}{t}}$, which implies $\norm{f} \le \sqrt{\binom{n+t}{t}} \norm{f}_2$.
\end{proof}

We recall a version of Haviland's theorem that characterize linear functionals that are represented by measures supported on a compact set.
\begin{theorem}[{\cite[th.17.3]{schmudgen_moment_2017}}]
  \label{measures_positive_polys}
  Let $S\subset\RR^n$ be compact and let $\pos(S)_t = \{f \in \RRg \mid \deg f \le t, \ f(x) \ge 0 \ \forall x \in  S \}$ . Then for a linear functional $L \in \RRg_t^*$,  $L \in \cM(S)^{[t]}$ if and only if $\braket{L}{f}\ge 0$ for all $f \in \pos(S)_t$.
\end{theorem}

We slightly modify \Cref{measures_positive_polys} in order to consider only polynomials of unit norm.
\begin{corollary}
\label{lem::positive_polys_variant}
  Let $P = \{ f \in \pos(S)_t \mid \norm{f}_2 = 1\}$ and let $L \in \RRg_t^*$. Then $L \in \cM(S)^{[t]}\subset \RRg_t^*$ if and only if $\braket{L}{f}\ge 0$ for all $f \in P$.
\end{corollary}
\begin{proof}
Notice that $\braket{L}{f}\ge 0 \iff \braket{L}{\frac{f}{\norm{f}_{2}}}\ge 0$. Then apply \Cref{measures_positive_polys}.
\end{proof}

We interpret \Cref{lem::positive_polys_variant} in terms of convex
geometry. The convex set
\[
  \cM(S)^{[t]} = \{ \, L \in \RRg_t^* \mid\forall f \in P, \braket{L}{f}\ge 0 \  \}
\]
is the convex cone dual to $P$.
Any $f\in P$ is defining an hyperplane $\braket{L}{f}= 0$ in
$\RRg_t^*$, and an associated halfspace $H_f=\{L \in
\RRg_{t}^{*}\mid \braket{L}{f}\ge 0 \}$ such that
$\cM(S)^{[t]} \subset H_f$.  \Cref{lem::positive_polys_variant} means
that $\cM(S)^{[t]} = \bigcap_{f\in P} H_f$.

We consider a relaxation of the positivity condition to prove our convergence.

\begin{definition}
\label{def::epsilon}
  For $\epsilon \ge 0$ and $P$ as in \Cref{lem::positive_polys_variant}, we define $C(\epsilon) = \{ L \in \RRg_t^* \mid\forall f\in P, \braket{L}{f} \ge -\epsilon\}$.
\end{definition}

Notice that by definition and \Cref{lem::positive_polys_variant} we have $C(0) = \cM(S)^{[t]}$.

We show now that $C(\epsilon)$ contains the truncated positive linear functionals of total mass one for a large enough order of the hierarchy.

\begin{lemma}
\label{lem::inclusion}
  Let $\ell \ge \gamma'(n,\vb g)\,  t^{3.5n\textit{\L}}\, \binom{n+t}{t}^{\frac{5n\textit{\L}}{4}} \epsilon^{-2.5n\textit{\L}}$, where $\vb g$ satisfy assumption \eqref{assum::norm} and $\gamma'(n, \vb g)$ is given by \Cref{eq::f_eta_bound}. Then $\cLone_\ell(\vb g)^{[t]} \subset  C(\epsilon)$.
\end{lemma}
\begin{proof}
  By \Cref{lem::norm_comparison}, for all $f \in P$ we have
  $\norm{f}\le  \binom{n+t}{t}^{1\over 2}$.
  From \Cref{thm::quad_module_weirstrass}, we deduce that for $\ell \ge \gamma'(n,\vb g)\,  t^{3.5n\textit{\L}}\, \binom{n+t}{t}^{\frac{5n\textit{\L}}{4}} \epsilon^{-2.5n\textit{\L}}$, we have $f-f^*+\epsilon = q \in
  \tqgen{g}{\ell}$. Thus for
  $L \in \cLone_\ell(\vb g)^{[t]}$ we obtain
  $\braket{L}{f+\epsilon} = \braket{L}{q + f^*}\ge
  0$. Therefore $\braket{L}{f} \ge -\epsilon$:
  this shows that $\cLone_\ell(\vb g)^{[t]} \subset  C(\epsilon)$.
\end{proof}

The convex set $C(\epsilon)$ can be seen as a \emph{tubular}
neighborhood of $\cM(S)^{[t]}$. We are going to bound its Hausdorff
distance to the measures. We state and prove the result in the general setting of convex geometry, and finally use it to prove \Cref{thm::moment_rate_convergence}.

\begin{lemma}
\label{lem::convex_bound}
  Let $C = \bigcap_{H \in \mathcal{H}} H$ be a closed convex set
  described as intersection of half spaces $H = \{ \vb x \in \R^N \mid
  c_H \cdot \vb x + b_H \ge 0\}$, where
\begin{itemize}
 \item $\norm{c_H}_{2} = 1$ for all $H \in \mathcal{H}$;
 \item $\mathcal{H}$ is the set of all the half-spaces containing $C$ (of unit normal).
\end{itemize}
If $H(\epsilon) = \{ \vb x \in \R^N \mid c_H \cdot \vb x + b_H \ge -\epsilon\}$ and $C(\epsilon) = \bigcap_{H \in \mathcal{H}} H(\epsilon)$, then $\hdist{C}{C(\epsilon)} \le \epsilon$.
\end{lemma}
\begin{proof}
  By definition $C \subset C(\epsilon)$. Assume that this inclusion is
  proper, otherwise there is nothing to prove, and let $\xi \in
  C(\epsilon) \setminus C$. Consider the closest point $\eta$ in $C$ of $\xi$ on $C$, and the half space $H = \{ \vb x\in \R^N \mid \frac{\eta - \xi}{\norm{\eta - \xi}_2} \cdot \vb x + b \ge 0 \} \in \mathcal{H}$
  defined by the affine supporting hyperplane orthogonal to $\eta - \xi$ passing through $\eta$ (and thus $\frac{\eta - \xi}{\norm{\eta - \xi}_2} \cdot \eta = - b$). Notice that $H \in \mathcal{H}$ since $H$ is defined by a normalized supporting hyperplane of $C$.

  Finally notice that $\norm{\eta - \xi}_2 = \frac{(\eta - \xi) \cdot
    (\eta - \xi)}{\norm{\eta - \xi}_2} = - \frac{\eta -
    \xi}{\norm{\eta - \xi}_2} \cdot \xi + \frac{\eta -
    \xi}{\norm{\eta - \xi}_2} \cdot \eta = - (\frac{\eta - \xi}{\norm{\eta - \xi}_2} \cdot \xi + b)$.
  Since $\xi \in C(\epsilon)$ and $H \in \mathcal{H}$, we have
  $(\frac{\eta - \xi}{\norm{\eta - \xi}_2} \cdot \xi + b) \ge
  -\epsilon$, and thus $0 < \norm{\eta - \xi}_2 \le \epsilon$. Then
  the distance between any $\xi \in C(\epsilon) \setminus C$ and its
  closest point $\eta \in C$ is $\le \epsilon$, which implies $\hdist{C}{C(\epsilon)} \le \epsilon$.
\end{proof}

\begin{theorem}
\label{thm::moment_rate_convergence}
  Let $\cQ(\vb g)$ be a quadratic module where $\vb g$ satisfy
  assumption \eqref{assum::norm} and let $$\ell \ge \gamma'(n,\vb g)\,  t^{3.5n\textit{\L}}\, \binom{n+t}{t}^{\frac{5n\textit{\L}}{4}} \epsilon^{-{2.5n\textit{\L}}}$$ with $\gamma'(n, \vb g)$ given by \Cref{eq::f_eta_bound}.
  Then $\hdist{\cM(S)^{[t]}}{\cLone_\ell(\vb g)^{[t]}} \le \epsilon$.
\end{theorem}
\begin{proof}
By \Cref{lem::positive_polys_variant} we have:
  \[
    \cM(S)^{[t]} = \{ \, L \in \RRg_t^* \mid\ \forall f \in P,
    \braket{L}{f}\ge 0  \}= \cap_{f\in P} H_{f},
  \]
where $H_{f}=\{ L \in \RRg_t^* \mid \braket{L}{f}\ge 0\}$
with $\norm{f}_{2}=1$ and $f \in \pos(S)_{t}$.
We check that the hyperplanes $H_{f}$ with $f\in P$ defining $\cM(S)^{[t]}$ satisfy the hypothesis of \Cref{lem::convex_bound}:
  \begin{itemize}
    \item The half-space $H_{f}$ has a unit normal since $\norm{f}_2 = 1$;
    \item Any supporting hyperplane of $\cM(S)^{[t]}$ defines an
      half-space $H_{f}=\{L\in \RRg_{t}^{*}\mid
      \braket{L}{f} \ge 0$ with $f\in
      P$. Indeed if $f$ defines a supporting hyperplane of $\cM(S)^{[t]}$, then
      $\braket{\mu}{f} = \int f \dd{\mu} \ge 0$ for all $\mu \in
      \cM(S)^{[t]}$. In particular for all $x \in S$ we have $f(x) =
      \int f \dd{\delta_x} \ge 0$
    (where $\delta_x$ denotes the dirac measure concentred at
    $x$). This proves that $f \in \pos(S)_t$ and, normalizing it, we can assume $f \in P$.
  \end{itemize}
  Then from \Cref{lem::convex_bound} we have
  $\hdist{\cM(S)^{[t]}}{C(\epsilon)} \le \epsilon$.

  Finally by \Cref{lem::inclusion} we deduce that  $\cLone_\ell(\vb
  g)^{[t]} \subset C(\epsilon)$ and conclude that
  \[
  \hdist{\cM(S)^{[t]}}{\cLone_\ell(\vb g)^{[t]}}\le \hdist{\cM(S)^{[t]}}{C(\epsilon)}  \le \epsilon.
  \]
\end{proof}

Notice that in \Cref{thm::moment_rate_convergence} we are bounding the
distance between normalized linear functionals and measures that may
be \emph{not} normalized (i.e. not a probability measure). In the
following we solve this problem.

We recall and adapt to our context \cite[lem.~3]{josz_strong_2016} to
obtain a bound on the norm of pseudo-moment sequences. In particular we do not assume that the ball constraint is an explicit inequality, but only that the quadratic module is Archimedean.


\begin{lemma}
\label{lem::moment_norm}
  Assume that $r^2 - \norm{\vb X}_2^2 = q \in \cQ_{\ell_0}(\vb g)$. Then for all $t\in \N$ and $\ell \ge 2t-2+\ell_0$, if $L \in \cLone_{\ell}(\vb g)$ we have $\norm{L^{[2t]}}_2 \le \sqrt{\binom{n+t}{t}}\sum_{k=0}^t r^{2k}$.
\end{lemma}
\begin{proof}
  For $L\in \cLone_{\ell}(\vb g)$, let $H_{L}^{k}$ be
  the Moment matrix of $L$ in degree $\le 2k$, which is
  semi-definite positive.
  Let $\norm{H_{L}^k}_F$ be its Frobenius norm, i.e.
  $\norm{H_{L}^k}_F = \sqrt{\sum_{\abs{\alpha}, \abs{\beta} \le
      k} L_{\alpha+\beta}^2}$, and $\norm{H_{L}^k}_2$ its
  $\ell^{2}$ operator norm, i.e. the maximal eigenvalue of $H_{L}^k$.
  Notice that by definition we have $\norm{L^{[2k]}}_2 \le \norm{H_{L}^k}_F$ and $\norm{H_{L}^k}_2 \le \sqrt{\trace H_{L}^k}$, Moreover recall $\norm{H_{L}^k}_F \le \sqrt{\rank (H_{L}^k)} \norm{H_{L}^k}_2$.
  To obtain a bound on $\norm{L^{[2k]}}_2$, we are going to use
  $\trace{H_{L}^k}= \sum_{\abs{\alpha} \le k} L_{2\alpha} =
  \braket{L^{[2k]}}{\sum_{\abs{\alpha} \le k}
    \vb{X}^{2\alpha}}$. As for $k\le t$,
\[
\displaystyle (r^2-\norm{\vb X}_2^2)(\sum_{\abs{\alpha} \le k-1}
\vb{X}^{2\alpha}) \in \cQ_{2t-2+\ell_0}(\vb g) \subset \cQ_\ell(\vb g).
\]
we have
\begin{align*}
    0 &\le \braket{L}{(r^2-\norm{\vb X}_2^2)(\sum_{\abs{\alpha} \le k-1} \vb{X}^{2\alpha})} =
     r^2\braket{L}{\sum_{\abs{\alpha} \le k-1} \vb{X}^{2\alpha}}
    - \braket{L}{\norm{\vb X}_2^2(\sum_{\abs{\alpha} \le k-1} \vb{X}^{2\alpha})} \\
    & = r^2\trace{H_{L}^{k-1}} -
      \big(\braket{L}{\sum_{\abs{\alpha} \le k} \vb{X}^{2\alpha}} -\braket{L}{1}\big) =
    r^2\trace{H_{L}^{k-1}} + 1 - \trace{H_{L}^{k}},
  \end{align*}
that is, $\trace{H_{L}^{k}} \le r^2\trace{H_{L}^{k-1}} +
1$. Since $\trace{H_{L}^0} = L_0 =1$, we deduce by induction
on $k$ that $\trace{H_{L}^{t}} \le \sum_{k=0}^t r^{2k}$ and thus
  \[
    \norm{L^{[2t]}}_2 \le \norm{H_{L}^t}_F \le \sqrt{\rank (H_{L}^t)} \norm{H_{L}^t}_2 \le \sqrt{\binom{n+t}{t}}\trace H_{L}^t\le \sqrt{\binom{n+t}{t}}\sum_{k=0}^t r^{2k}.
  \]
\end{proof}
Finally we are ready to prove \Cref{thm::moment_rate_convergence_prob}, where we obtain the bound of the distance between normalized linear functionals and probability measures.
\begin{proof}[{Proof of \Cref{thm::moment_rate_convergence_prob}}]
  Let $\epsilon'= {1 \over 2}\epsilon t^{-1} \binom{n+t}{t}^{-{1\over
      2}} \le {1\over 4}$,
   $L \in \cLone_\ell(\vb g)^{[2t]}$ and $\mu \in \cM(S)^{[2t]}$ be
   the closest point to $L$.
  We first bound the norm of $\mu$.
As  \[\ell \ge \gamma(n,\vb g)\, 6^{2.5n\textit{\L}}\, t^{6n\textit{\L}}\, \binom{n+t}{t}^{\frac{5n\textit{\L}}{2}}
\epsilon^{-{2.5n\textit{\L}}} = \gamma'(n,\vb g)\,  t^{3.5n\textit{\L}}\, \binom{n+t}{t}^{\frac{5n\textit{\L}}{4}}
(\epsilon')^{-{2.5n\textit{\L}}} ,\] by \Cref{thm::moment_rate_convergence} we have $\dist{L}{\mu} \le \epsilon'$.

  Let $\mu_0 = \int 1 \dd{\mu}$. We want to bound the distance between $L$ and $\frac{\mu}{\mu_0} \in \cMone(S)^{[2t]}$. Notice that
  \begin{equation}
    \label{eq::first_bound}
    \dist{L}{\frac{\mu}{\mu_0}} \le \dist{L}{\mu} +\dist{\mu}{\frac{\mu}{\mu_0}} \le \epsilon' + \abs{\frac{1-\mu_0}{\mu_0}}\norm{\mu}_2.
  \end{equation}

  Since $L_0 = 1$, $\dist{L}{\mu} \le \epsilon'$ implies
  $1-\epsilon'\le \mu_0 \le 1+\epsilon'$, and therefore
  $\abs{\frac{1-\mu_0}{\mu_0}} \le
  \frac{\epsilon'}{1-\epsilon'}$. Moreover, using
  \Cref{lem::moment_norm} we have
  \[
  \norm{\mu}_2 = \norm{\mu - L + L} \le \dist{\mu}{L} + \norm{L}_2 \le \epsilon' + (t+1)\sqrt{\binom{n+t}{t}}.
  \]
  Then from \Cref{eq::first_bound} we conclude that
  \[
    \dist{L}{\frac{\mu}{\mu_0}} \le \epsilon' +
    \frac{\epsilon'}{1-\epsilon'} (\epsilon' +
    (t+1)\sqrt{\binom{n+t}{t}}) = \frac{\epsilon'}{1-\epsilon'} +
    \frac{\epsilon'}{1-\epsilon'} (t+1)\sqrt{\binom{n+t}{t}}
    \le
    2  \epsilon' t \sqrt{\binom{n+t}{t}} = \epsilon,
  \]
  since $\epsilon'\le {1\over 4}$, $n\ge 1$ and $t\ge 1$.
\end{proof}

\begin{corollary}\label{cor::moment_rate_convergence_prob}
With the hypothesis of \Cref{thm::moment_rate_convergence_prob} and the CQC (\Cref{def::CQC}) satisfied at every point of $\cS(\vb g)$, then
 \[
   \hdist{\cMone(S)^{[2t]}}{\cLone_\ell(\vb g)^{[2t]}}\le \epsilon
 \]
 if $\ell \ge \gamma(n,\vb g)\, 6^{2.5n}\, t^{6n}\, \binom{n+t}{t}^{2.5n\textit{\L}}
 \epsilon^{-{2.5n}}$.
\end{corollary}
\begin{proof}
    Apply \Cref{thm::moment_rate_convergence_prob} and \Cref{thm::L_equal_1}.
\end{proof}
In \Cref{thm::moment_rate_convergence_prob} we prove a bound for the
convergence of Lasserre truncated pseudo-moments to moments of
measures. The convergence, without bounds, can be deduced from
\cite[th. 3.4]{schweighofer_optimization_2005} by taking as objective
function a constant.  On the other hand, we can deduce
\cite[th.~3.4]{schweighofer_optimization_2005} from
\Cref{thm::moment_rate_convergence_prob}, by considering the sections
of $\cLone_\ell(\vb g)^{[t]}$ given by
$\braket{L}{f} = f^*_{\mom, k}$.

In the context of Generalized Moment Problems (GMP), general
convergence to moments of measures has been studied in
\cite{tacchi_convergence_2021}. The uniform bounded mass assumption in
\cite{tacchi_convergence_2021} is trivially satisfied in the context
of Polynomial Optimization, since $L_0 = \braket{L}{1} = 1$:
the convergence result of \cite{tacchi_convergence_2021} is thus more
general than \cite[th.~3.4]{schweighofer_optimization_2005} and the
one implied by \Cref{thm::moment_rate_convergence_prob}. But we
conjecture, and leave it for future exploration, that it is possible
to extend the proof technique of
\Cref{thm::moment_rate_convergence_prob} to the GMP and give bounds on
the rate of convergence also in this extended context.

\paragraph{Aknowkedgments.}
The authors thank M. Laurent and L. Slot for the discussion about
Schmüdgen's theorem on $\ubox$, A. Parusiński and K. Kurdyka for the useful
suggestions on the Łojasiewicz and Markov inequalities and
F. Kirschner for discussions on half space descriptions of convex bodies. The authors thank the anonymous referees for their suggestion, that helped improving the presentation and pointed out errors present in previous versions of the article.
\printbibliography
\end{document}

%% file: preamble.tex
\usepackage{amsmath,amssymb,amsthm}
\usepackage{mathtools}
\usepackage{amsfonts}
\usepackage{physics}

\usepackage[T1]{fontenc}
\usepackage[utf8]{inputenc}
\usepackage[english]{babel}
\usepackage{kpfonts}
\usepackage{microtype}
\usepackage[autostyle]{csquotes}

\usepackage[style=alphabetic, url=false, doi=false, date=year,backend=bibtex, maxbibnames=10]{biblatex}
\AtEveryBibitem{\clearfield{issn}}
\AtEveryCitekey{\clearfield{issn}}
\AtEveryBibitem{\clearfield{note}}
\AtEveryCitekey{\clearfield{note}}

\usepackage{enumitem}
\usepackage{hyperref}
\usepackage{tabularx}

\usepackage[english, algosection, algoruled, noline]{algorithm2e}

\usepackage{cleveref}
\usepackage{xcolor}

\usepackage[margin=2.25cm]{geometry}

 \theoremstyle{plain}
 \newtheorem{theorem}{Theorem}[section]
 \newtheorem{lemma}[theorem]{Lemma}
 \newtheorem{proposition}[theorem]{Proposition}
 \newtheorem{corollary}[theorem]{Corollary}

 \theoremstyle{definition}
 \newtheorem{definition}[theorem]{Definition}
 \newtheorem{example}[theorem]{Example}

 \theoremstyle{remark}
 \newtheorem*{remark}{Remark}

\DeclarePairedDelimiter{\floor}{\lfloor}{\rfloor}

\newcommand{\N}{\mathbb N}

\newcommand{\R}{\mathbb R}

\renewcommand{\epsilon}{\varepsilon}

\newcommand{\cS}{\mathcal{S}}
\newcommand{\cO}{\mathcal{O}}
\newcommand{\cQ}{\mathcal{Q}}

\newcommand{\cL}{\mathcal{L}}
\newcommand{\cLone}{\mathcal{L}^{(1)}}

\renewcommand{\cS}{\mathcal{S}}
\newcommand{\cM}{\mathcal{M}}
\newcommand{\cMone}{\mathcal{M}^{(1)}}

\newcommand*{\tqgen}[2]{\cQ_{{#2}}({\vb{#1}})}

\newcommand*{\hdist}[2]{\hdistance ({#1},{#2})}
\newcommand*{\dist}[2]{\distance ({{#1},{#2}})}

\DeclareMathOperator{\pos}{Pos}
\DeclareMathOperator{\distance}{d}
\DeclareMathOperator{\hdistance}{d_{H}}

\DeclareMathOperator{\jac}{Jac}

\def \sos {\mathrm{SoS}}
\def \mom {\mathrm{MoM}}

\def \ubox {[-1,1]^n}
\setlist[enumerate,1]{label={(\roman*)},ref={\thetheorem (\roman*)}}

\def\RR{\mathbb{R}}

\def\RRg{\RR[\vb X]}

\def\RRT{\RR[T]}

\newcommand{\assign}{:=}
\newcommand{\cdummy}{\cdot}